\documentclass[11pt]{amsart}

\usepackage{amsmath}
\usepackage{amssymb,mathtools,amsthm,latexsym}
\usepackage{mathrsfs}
\usepackage{bm}
\usepackage{url}
\usepackage{xcolor}
\usepackage{todonotes}
\numberwithin{equation}{section}

\calclayout

\usepackage{comment}
\usepackage{soul}

\newtheorem{thm}{Theorem}[section]
\newtheorem{lemma}[thm]{Lemma}
\newtheorem{cor}[thm]{Corollary}
\newtheorem{prop}[thm]{Proposition}
\theoremstyle{definition}
\newtheorem{defn}[thm]{Definition}

\newtheorem{remark}[thm]{Remark}

\title[Tensorization of $p$-weak differential structures]{Tensorization of $p$-weak Differentiable structures}
\author{Sylvester Eriksson-Bique}
\author{Tapio Rajala}
\author{Elefterios Soultanis}

\address{Sylvester Eriksson-Bique\\
Research Unit of Mathematical Sciences \\
P.O.Box 3000\\
FI-90014 Oulu\\
Finland}
\email{\tt sylvester.eriksson-bique@oulu.fi}

\address{Tapio Rajala\\
Department of Mathematics and Statistics \\
University of Jyvaskyla\\
P.O. Box 35 (MaD) \\
         FI-40014 University of Jyvaskyla \\
         Finland}
\email{\tt tapio.m.rajala@jyu.fi}

\address{Elefterios Soultanis\\
IMAPP\\
Radboud University, 6525 AJ Nijmegen, Heyendaalsweg 135, The Netherlands}
\email{\tt elefterios.soultanis@gmail.com}

\def\Xint#1{\mathchoice
	{\XXint\displaystyle\textstyle{#1}}%
	{\XXint\textstyle\scriptstyle{#1}}%
	{\XXint\scriptstyle\scriptscriptstyle{#1}}%
	{\XXint\scriptscriptstyle\scriptscriptstyle{#1}}%
	\!\int}
\def\XXint#1#2#3{{\setbox0=\hbox{$#1{#2#3}{\int}$}
		\vcenter{\hbox{$#2#3$}}\kern-.5\wd0}}

\def\dashint{\Xint-}

\newcommand{\N}{\ensuremath{\mathbb{N}}}
\newcommand{\R}{\ensuremath{\mathbb{R}}}

\def\vint_#1{\mathchoice%
        {\mathop{\kern 0.2em\vrule width 0.6em height 0.69678ex depth -0.58065ex
                \kern -0.8em \intop}\nolimits_{\kern -0.4em#1}}%
        {\mathop{\kern 0.1em\vrule width 0.5em height 0.69678ex depth -0.60387ex
                \kern -0.6em \intop}\nolimits_{#1}}%
        {\mathop{\kern 0.1em\vrule width 0.5em height 0.69678ex
            depth -0.60387ex
                \kern -0.6em \intop}\nolimits_{#1}}%
        {\mathop{\kern 0.1em\vrule width 0.5em height 0.69678ex depth -0.60387ex
                \kern -0.6em \intop}\nolimits_{#1}}}
\def\vintslides_#1{\mathchoice%
        {\mathop{\kern 0.1em\vrule width 0.5em height 0.697ex depth -0.581ex
                \kern -0.6em \intop}\nolimits_{\kern -0.4em#1}}%
        {\mathop{\kern 0.1em\vrule width 0.3em height 0.697ex depth -0.604ex
                \kern -0.4em \intop}\nolimits_{#1}}%
        {\mathop{\kern 0.1em\vrule width 0.3em height 0.697ex depth -0.604ex
                \kern -0.4em \intop}\nolimits_{#1}}%
        {\mathop{\kern 0.1em\vrule width 0.3em height 0.697ex depth -0.604ex
                \kern -0.4em \intop}\nolimits_{#1}}}

\newcommand{\LIP}{\ensuremath{\mathrm{LIP}}}

\newcommand{\Lip}{\ensuremath{\mathrm{Lip}}}

\newcommand{\defeq}{\mathrel{\mathop:}=}

\newcommand{\Mod}{\ensuremath{\mathrm{Mod}}}

\newcommand{\eps}{\varepsilon}

\newcommand{\ud}{\mathrm{d}}
\newcommand{\Ne}[2]{N^{1,#1}(#2)}
\newcommand{\inv}{^{-1}}

\newcommand{\spt}{\mathrm{spt}}

\begin{document}

\begin{abstract}
We consider $p$-weak differentiable structures that were recently introduced in \cite{teriseb}, and prove that the  product of $p$-weak charts is a $p$-weak chart. This implies that the product of two spaces with a $p$-weak differentiable structure also admits a $p$-weak differentiable structure. We make partial progress on the tensorization problem of Sobolev spaces by showing an isometric embedding result. Further, we establish tensorization when one of the factors is PI.
\end{abstract}
\maketitle
\tableofcontents

\section{Introduction}

\subsection{Background}

A weak notion of differentiable charts in singular metric spaces first arose in Cheeger's seminal paper \cite{che99}. \emph{Lipschitz differentiability charts}, or Cheeger charts, have since become ubiquitous in analysis in metric spaces, with connections to rectifiability, (non-)embedding results and other topics in geometric measure theory. A Cheeger chart $(U,\varphi)$ consists of a Borel set $U\subset X$ and a Lipschitz map $\varphi\colon X\to \R^n$ such that every $f\in\LIP(X)$ admits a differential $x\mapsto \ud_xf\colon U\to (\R^n)^*$, uniquely determined for $\mu$-a.e. $x\in U$, satisfying

\begin{equation}\label{eq:che_diff}
f(y)-f(x)=\ud_xf((\varphi(y)-\varphi(x))+o(d(x,y)).
\end{equation}

Cheeger charts describe the infinitesimal behaviour of Lipschitz functions but are not well suited for studying Sobolev functions in the absence of additional assumptions. In \cite{teriseb} the first and third named authors developed \emph{$p$-weak charts}, a further weakening of Cheeger charts. We refer to Section 2 for their definition and mention here that $p$-weak charts control the behaviour of Sobolev functions curvewise  and exist under very mild assumptions, e.g. when the underlying space has finite Hausdorff dimension. While weaker than the notion of Cheeger charts, the existence of non-trivial $p$-weak charts guarantees the existence of non-negligible families of curves, and induces a pointwise norm given as the essential supremum of directional derivatives along these curves. Indeed, Sobolev functions admit a $p$-weak differential with respect to $p$-weak charts, and the minimal upper gradient is recovered as the pointwise norm of the differential.

In this paper we consider products of spaces admitting a $p$-weak differentiable structure and prove that they likewise admit a $p$-weak differentiable structure. 

\subsection{Statement of main results}

Let $X=(X,d_X,\mu)$ and $Y=(Y,d_Y,\nu)$ be two metric measure spaces, i.e. complete separable metric spaces equipped with  Radon measures that are finite on balls. Throughout this paper we equip the product space $X\times Y$ with the product measure $\mu\times\nu$ and metric
\begin{equation}\label{eq:prod_metric}
    d((x,y),(x',y'))\defeq \|(d_X(x,x'),d_Y(y,y'))\|,\quad (x,y),(x',y')\in X\times Y,
\end{equation}
where $\|\cdot\|$ is a norm on $\R^2$. Note that the behaviour of the norm on the first quadrant determines the metric. In the theorem below we denote by 
\[
\|(x,y)\|'\defeq \max\{ |ax+by|: \|(a,b)\|=1,\ a,b\ge 0 \},\quad (x,y)\in \R^2,
\]
the \emph{partial dual norm} of a given planar norm $\|\cdot\|$. While $\|\cdot\|'\le \|\cdot\|^*$ in general, the equality $\|\cdot\|'=\|\cdot\|^*$ holds for $l^p$-norms (and more generally norms satisfying $\|(a,b)\|=\|(|a|,|b|)\|$ for $(a,b)\in \R^2$). Roughly speaking the need for the partial (rather than the ``full'') dual norm comes from the fact that the metric speed of product curves does not distinguish between the direction in which each of the component curves is traversed, see estimate \eqref{eq:norm_is_min_ug} in the proof of Proposition \ref{prop:ptwise_norm_on_prod}.



\begin{thm}\label{thm:tensor-charts} Suppose that $(U,\varphi)$ and $(V,\psi)$ are $p$-weak charts of dimension $N$ and $M$ in $X$ and $Y$, respectively. Then $(U\times V, \varphi\times \psi)$ is an $(N+M)$-dimensional $p$-weak chart for $X\times Y$. The local norm on $(\R^{N})^*\times(\R^M)^*\equiv (\R^{N+M})^*$ is given by
\begin{align*}
|(\xi_X,\xi_Y)|_{(x,y)} = \|(|\xi_X|_x, |\xi_Y|_y)\|',\quad (\xi_X,\xi_Y)\in (\R^{N+M})^*,    
\end{align*}
for $\mu\times\nu$-almost every $(x,y)\in X\times Y$.
\end{thm}

Theorem \ref{thm:tensor-charts} yields the following immediate corollary.
\begin{cor}\label{cor:prod_has_diff_struct}
Let $X$ and $Y$ be two metric measure spaces which admit a $p$-weak differentiable structure. Then $X\times Y$ with the product metric \eqref{eq:prod_metric} admits a $p$-weak differentiable structure.
\end{cor}

In proving Theorem \ref{thm:tensor-charts} we will use a characterization of $p$-weak charts in terms of existence and uniqueness of differentials in the spirit of \eqref{eq:che_diff}. In the following definition, let $U\subset X$ be a Borel set with $\mu(U)>0$ and $\varphi\in N^{1,p}_{loc}(X;\R^N)$. 
\begin{defn}
 A Borel map $\bm\xi\colon U\to\R^N$ is a $p$-weak differential of a function $f\in \Ne pX$ with respect to $(U,\varphi)$, if 
\begin{align}\label{eq:differential}
(f\circ\gamma)_t'=\bm\xi_{\gamma_t}((\varphi\circ\gamma)_{t}'),\quad a.e.\ t\in \gamma\inv(U)
\end{align}
for $p$-a.e. curve $\gamma$ in $X$.
\end{defn}

Note that there is no uniqueness condition imposed in the definition above. We say that $f\in \Ne pX$ admits a \emph{unique} $p$-weak differential with respect to $(U,\varphi)$, if the map $\bm\xi$ in \eqref{eq:differential} is unique in the following sense: if $\bm\xi'\colon U\to (\R^N)^*$ is another Borel map satisfying \eqref{eq:differential} for $p$-a.e. $\gamma$, then $\bm\xi=\bm\xi'$ $\mu$-a.e. on $U$.

\begin{thm}\label{thm:equiv_lipdiff_weakchart}
	Let $U\subset X$ be a Borel set with $\mu(U)>0$ and $\varphi\colon X\to \R^N$ a Lipschitz map. The following are equivalent.
	\begin{itemize}
	    \item[(1)] $(U,\varphi)$ is a $p$-weak chart;
	    \item[(2)] every $f\in \Ne pX$ admits a unique differential $\ud f\colon U\to (\R^N)^*$ with respect to $(U,\varphi)$;
	    \item[(3)] every $f\in \LIP(X)$ admits a unique differential $\ud f\colon U\to (\R^N)^*$ with respect to $(U,\varphi)$.
	\end{itemize} 
\end{thm}

\subsection{Application: Tensorization problem of Sobolev spaces}
We use Theorem \ref{thm:tensor-charts} and Corollary \ref{cor:prod_has_diff_struct} to make partial progress on the ``tensorization problem'' for Sobolev spaces. The tensorization problem first appeared in \cite{ags14b} and was formulated using Sobolev spaces defined via test plans, which we denote here by $W^{1,p}$, see Definition \ref{def:sob_space}. It has later been investigated e.g. in \cite{amb-pin-spe15}. Given $p\ge 1$, the \emph{Beppo--Levi} space $J^{1,p}(X,Y)$ consists of Borel functions $f\in L^p(X\times Y)$, which satisfy the following:
\begin{itemize}
	\item[(a1)] For $\mu_X$-a.e. $x\in X$, $f_x:=f(x,\cdot)\in W^{1,p}(Y)$;
	\item[(a2)] For $\mu_Y$-a.e. $y\in Y$, $f^y:=f(\cdot,y)\in W^{1,p}(X)$, and 
	\item[(a3)] $\displaystyle \int_{X\times Y} (|Df^y|_p^p(x)+|Df_x|_p^p(y))\ud\mu(x)\ud\nu(y)<\infty.$
\end{itemize}
We refer to the Appendix for the measurability of the integrand in (a3). Observe that, although the Sobolev space $W^{1,p}$ and $N^{1,p}$ are isometrically isomorphic, it is not trivial that one can replace $W^{1,p}$ by $N^{1,p}$ in (a1) and (a2), a point we address in Section 5. The space  $J^{1,p}(X,Y)$ equipped with the norm $\|f\|_{J^{1,p}}\defeq (\|f\|_{L^p}+[f]_{J^{1,p}}^p)^{1/p}$ is a Banach space, where $[f]_{J^{1,p}}$ is the seminorm
\begin{equation}\label{eq:j1p_norm}
[f]_{J^{1,p}}=\left(\int_{X\times Y} (\|(|Df^y|_p(x),|Df_x|_p(y))\|')^p\ud\mu(x)\ud\nu(y)\right)^{1/p}.
\end{equation}
The ``tensorization problem'' of Sobolev spaces asks whether or not the equality
\begin{equation}\label{eq:tensorprob}
W^{1,p}(X\times Y)=J^{1,p}(X,Y)
\end{equation}
holds. Roughly speaking this amounts to asking whether knowledge of the directional derivatives in the $X$ and $Y$ directions is enough to ensure Sobolev regularity in $X\times Y$. While the inclusion $W^{1,p}(X\times Y)\subset J^{1,p}(X,Y)$ is elementary not much else is known without additional assumptions. 

\begin{thm}\label{thm:tensor_sob}
Suppose $X$ and $Y$ admit a $p$-weak differentiable structure and let $f\in \Ne p{X\times Y}$. Then
\begin{align*}
\ud f_{(x,y)}=(\ud_xf^y,\ud_yf_x)
\end{align*}
and
\begin{align*}
|Df|_p(x,y)=\|(|Df^y|_p(x),|Df_x|_p(y))\|'
\end{align*}
for $\mu\times\nu$-a.e. $(x,y)$. In particular, the embedding $W^{1,p}(X\times Y)\subset J^{1,p}(X,Y)$ is isometric. 
\end{thm}

Theorem \ref{thm:tensor_sob} leaves open the question of equality in \eqref{eq:tensorprob} but shows that the elementary inclusion $W^{1,p}(X\times Y)\subset J^{1,p}(X,Y)$ is isometric (e.g. when the spaces are finite dimensional), providing partial evidence in favour of the tensorization property.

While a full solution to the tensorization problem remains open, we are able to establish \eqref{eq:tensorprob} under the additional assumption that one of the factors is a PI-space. Recall that a $p$-PI space is a complete doubling metric measure space supporting a weak $p$-Poincar\'e inequality, cf. \cite{HKST07}.

\begin{thm}\label{thm:Onefactor} Suppose that $X$ is a $p$-PI-space and $Y$ admits a $p$-weak differential structure. Then $W^{1,p}(X\times Y)=J^{1,p}(X,Y)$ and the norms coincide.
\end{thm}

Previously the same conclusion was (essentially) known to hold if \emph{both} factors are PI-spaces, see \cite[Theorem 3.4]{amb-pin-spe15} for a proof in the case $p=2$.

\subsection{Acknowledgements}

The first author was partially supported by the Finnish Academy under Research postdoctoral Grant No. 330048. The second author was partially supported by the Finnish Academy, Grant No. 314789. The third author was supported by the Swiss National Science Foundation Grant 182423. The authors thank Jeff Cheeger, Nicola Gigli, Enrico Pasqualetto, and Nageswari Shanmugalingam for many enlightening discussions. 

\section{Preliminaries}

\subsection{Product spaces}
Let $X$ and $Y$ be complete separable metric spaces, and $\|\cdot\|$ a norm on $\R^2$. Throughout this paper we will use the product metric on $X\times Y$ defined by 
\begin{align*}
d((x,y),(x',y'))\defeq \|(d_X(x,x'),d_Y(y,y'))\|,\quad (x,y),(x',y')\in X\times Y.
\end{align*}
Given an absolutely continuous curve $\gamma=(\alpha,\beta)$ in $X\times Y$, the metric speed satisfies
\begin{align*}
    |(\alpha,\beta)_t'|=\|(|\alpha_t'|,|\beta_t'|)\|\quad a.e.\ t,
\end{align*}
where $|\alpha_t'|$ and $|\beta_t'|$ are the metric speeds of the curves $\alpha$ and $\beta$ with respect to the metrics $d_X$ and $d_Y$, respectively.

\emph{Horizontal} curves are curves whose $Y$-component is constant, and their collection is denoted $H([0,1];X\times Y)$. Similarly, \emph{vertical} curves have constant $X$-component, and their collection is denoted $V([0,1];X\times Y)$. In Sections 5 and 6 we will need the notion of \emph{horizontal-vertical} curves which are obtained as concatenations of horizontal and vertical curves. The formal definition is given below.

\begin{defn}\label{def:hv_curve}
 A curve $\gamma = (\alpha,\beta) \in AC([0,1];X\times Y)$ is called a \emph{hv-curve} if there exists a partition of $[0,1]$ given as $0 = t_0 < t_1 < \cdots < t_k = 1$ so that for every $i \in \{1,\dots,k\}$
 either $\alpha$ or $\beta$ is constant in $[t_{i-1},t_i]$. We denote the set of hv-curves by $HV([0,1];X\times Y)$
\end{defn}
Given a HV-curve $\gamma=(\alpha,\beta)$ its constant speed parametrization $\bar \gamma=(\bar\alpha,\bar\beta)$ is also a HV-curve with the additional property that $\bar\gamma|_I$ is non-constant on any open interval $I\subset [0,1]$ unless $\gamma$ is a constant curve. For (non-constant) $\bar\gamma$ we may find a decomposition $0<t_1\ldots<t_n<1$, with either $\alpha$ or $\beta$ constant on each $[t_{i-1},t_i]$, which is maximal in the following sense: for any $i=1,\ldots,n$ and open $I\supset [t_{i-1},t_i]$, both $\bar\alpha|_I$ and $\bar\beta|_I$ are non-constant (i.e. none of the intervals $[t_{i-1},t_i]$ can be enlarged while keeping one of the component curves constant).  By convention constant curves have an empty decomposition.
\begin{defn}\label{def:hvn_curve}
Given a curve $\gamma$, we call the unique decomposition $\bm t(\gamma)=\{t_1<\ldots < t_n\}$ as above the collection of turning times of $\gamma$. For $n\in\N$, we denote by $HV_n$ the subset of $HV$ consisting of curves $\gamma$ such that the constant speed parametrization $\bar\gamma$ has exactly $n$ turning times.
\end{defn}

\subsection{Partial upper gradients}
We refer the reader to \cite{HKST07,bjo11} for a good account of modulus, line integrals along curves, weak upper gradients, and Newton-Sobolev spaces, and omit their definitions here. 

Let $\Gamma\subset AC([0,1];X)$ be a family of curves in a metric measure space $X$, and let $p\ge 1$. We say that a Borel function $g\colon X\to [0,\infty]$ is an upper gradient of a Borel function $f\colon X\to \R$ \emph{along $\Gamma$}, if the upper gradient inequality
\begin{equation}\label{eq:ug}
|f(\gamma_t)-f(\gamma_s)|\le \int_{\gamma|_{[s,t]}} g\ud s,\quad 0\le s<t\le 1
\end{equation}
holds for every $\gamma\in \Gamma$. We moreover say that $g$ is a $p$-weak upper gradient along $\Gamma$ if \eqref{eq:ug} holds for $p$-a.e. curve  $\gamma\in \Gamma$. If $\Gamma=AC([0,1];X)$ we say that $g$ is an upper gradient (resp. $p$-weak upper gradient) of $f$. 


We record the following result which establishes the existence of minimal weak partial upper gradients. 
\begin{prop}\label{prop:minimal_partial_ug}
    Let $\Gamma\subset AC([0,1];X)$ be a family of curves and $f\colon X\to\R$ a locally integrable function which admits a $p$-weak upper gradient $g\in L^p_{loc}(\mu)$ along $\Gamma$. Then there exists a \emph{minimal} $p$-weak upper gradient $|Df|_{p,\Gamma}$ of $f$ along $\Gamma$.
\end{prop}
Minimality in the claim above is intended in the sense that (a) $|Df|_{p,\Gamma}$ is a $p$-weak upper gradient of $f$ along $\Gamma$, and (b) $|Df|_{p,\Gamma}\le g$ for any locally $p$-integrable $p$-weak upper gradient $g$ of $f$. 
\begin{proof}
By Fuglede's Lemma, the family of partial weak upper gradients of $f$ is closed under (local) $L^p$-convergence. The collection of partial weak upper gradients of $f$ has the lattice property by argument in the proof of \cite[Lemma 6.3.14]{HKST07}. The existence of a minimal element in the lattice follows as in \cite[Theorem 6.3.20]{HKST07}.
\end{proof}

\subsection{Plans}
A \emph{plan} on $X$ is a finite measure $\bm\eta$ on $C([0,1];X)$ concentrated on absolutely continuous curves. 
The \emph{barycenter} $\bm\eta^\#$ of $\bm\eta$ is the measure on $X$ defined by 
\[
\bm\eta^\#(E):=\int\int_0^1\chi_E(\gamma_t)|\gamma_t'|\ud t\ud\bm\eta(\gamma),\quad E\textrm{ Borel}.
\]
If $\bm\eta^\#$ is an $L^q$-function -- i.e. $\bm\eta^\#=\rho\mu$ for some $\rho\in L^p(\mu)$ -- we say $\bm\eta$ is a \emph{$q$-plan}. To define test plans, denote
\[
e_t\colon C([0,1];X)\to X,\quad \gamma\mapsto \gamma_t,
\]
for fixed $t\in [0,1]$. We say that $\bm\eta$ is a $q$-test plan, if
\begin{equation}\label{eq:test_plan}
\int\int_0^1|\gamma_t'|^q\ud t\ud\bm\eta(\gamma)<\infty,
\end{equation}
and there exists $C>0$ such that $e_{t\ast}\bm\eta\le C\mu$ for each $t\in [0,1]$. If $q=\infty$, we replace \eqref{eq:test_plan} by the requirement that $\bm\eta$ is concentrated on a family of $L$-Lipschitz curves, for some $L$. Test plans appear in the definition of the Sobolev spaces $W^{1,p}$.

\begin{defn}\label{def:sob_space}
Let $p\ge 1$ and let $q$ be the dual exponent of $p$. A function $f\in L^p(\mu)$ belongs to the Sobolev space $W^{1,p}(X)$ if there exists $g\in L^p(\mu)$ such that 
\begin{align*}
\int|f(\gamma_1)-f(\gamma_0)|\ud\bm\eta\le \int\int_\gamma g\ud s\ud\bm\eta
\end{align*}
for every $q$-test plan $\bm\eta$ on $X$.
\end{defn}
We remark that the equality $\Ne pX=W^{1,p}(X)$ holds if properly interpreted, see e.g. \cite[Theorem 2.5]{teriseb}. Next we define the restriction operation on plans.
\begin{defn}[Restriction of a plan]\label{def:restr_of_plan}
 Let $\bm\eta$ be a plan on $X$ and $t,s \in [0,1]$, $t \le s$. The \emph{restriction of $\bm\eta$ to $[0,1]$}  is defined as the plan
 \[
  \bm\eta|_{[t,s]} = e_{[t,s]\ast}\bm\eta,
 \]
where $e_{[t,s]} \colon C(I;X) \to C(I;X)$ is the \emph{restriction map} satisfying for all $\gamma \in C(I;X)$ the equality
$e_{[t,s]}(\gamma) = \tilde\gamma$ with
$\tilde\gamma(r) = \gamma((1-r)t+rs)$.
\end{defn}

We record the following lemma which can be established by elementary arguments. We omit the proof.

\begin{lemma}
 Let $\bm\eta$ be a $q$-plan and $t,s \in [0,1]$, $t \le s$. Then $\bm\eta|_{[t,s]}$ is a $q$-plan. If $\bm\eta$ is a $q$-test plan then $\bm\eta|_{[s,t]}$ is a $q$-test plan.
\end{lemma}

In Section 5 we define the concatenation of plans which is, in a sense, an opposite operation to restricting plans.

\subsection{Disintegration} Disintegration of measures with respect to a Borel map is a far reaching generalization of Fubini's theorem. Although we will mostly apply it to plans and with respect to the evaluation map, we present a more general formulation below. The following theorem can be found in \cite[Theorem 5.3.1]{AGS08}.
\begin{thm}\label{thm:disintegration}
Let $\phi\colon X\to Y$ be a Borel map between complete separable metric spaces. Let $\bm\pi\in \mathcal P(X)$ and $\nu\defeq \phi_\ast\bm\pi$. Then there exists a $\nu$-a.e. $x\in X$ uniquely defined family of measures $\{\bm\pi_x\}\subset \mathcal P(Y)$ such that $\bm\pi_x$ is concentrated on $\phi\inv(x)$ and
\begin{align*}
\int G\ud\bm\pi=\int_X\left(\int_{\phi\inv(x)}G\ud\bm\pi_x\right)\ud\nu(x)
\end{align*}
for every Borel map $G\colon X\to [0,\infty]$.
\end{thm}

The disintegration is often used for the measure $\ud\bm\pi:=|\gamma_t'|\ud t\ud\bm\eta$ on $[0,1]\times AC([0,1];X)$ and the evaluation map $e:[0,1]\times AC([0,1];X)\to X$ given by $e(t,\gamma)=\gamma_t$, when $\bm\eta$ is a $q$-plan. This yields a family of measures $\{ \bm{\pi}_x \}$, which are $e_*\ud\bm\pi$-almost everywhere uniquely defined for $x\in X$. We call $\{ \bm{\pi}_x \}$ the \emph{disintegration of $\bm\pi$} (without reference to the map). Note here that $\bm\eta^\#=e_*\ud\bm\pi$.

\subsection{$p$-Weak differentiable structure} 

Given a Borel set $U\subset X$ of positive measure and $\varphi\in N^{1,p}_{loc}(X;\R^N)$, we say that $(U,\varphi)$ is \emph{$p$-independent} if 
\begin{align*}
\inf_{v\in D}|D(v\cdot\varphi)|_p(x)>0\quad \mu-\textrm{a.e. on }U
\end{align*}
for some (and thus any) countable dense subset $D\subset S^{N-1}$. The pair $(U,\varphi)$ is said to be \emph{$p$-maximal} if, for all Lipschitz maps $\psi\in \LIP(X;\R^M)$ with $M>N$ and Borel sets $V\subset U$ of positive measure, the pair $(V,\psi)$ is not $p$-independent.
\begin{defn}\label{def:p-weak_chart}
A pair $(U,\varphi)$ is a $p$-weak chart if it is both $p$-independent and $p$-maximal. 
\end{defn}
To describe the pointwise norm associated to $p$-weak charts we record the following result which will be useful in the sequel. In the statement $U\subset X$ is Borel and $\varphi\in N^{1,p}_{loc}(X;\R^N)$.

\begin{thm}[Lemmas 4.1--4.3 in \cite{teriseb}]\label{thm:canonical_repr_of_grad}
There exists a $q$-plan $\bm\eta$ on $X$ and a Borel set $D\subset X$ with $\mu\llcorner_D\ll \bm\eta^\#$, and $\{ \bm{\pi}_x \}$ the disintegration of $\ud\bm\pi:=|\gamma_t'|\ud t\ud\bm\eta$, such that
\begin{align*}
\Phi(x,\xi)\defeq \chi_D(x)\left\|\frac{(\varphi\circ\gamma)_t'}{|\gamma_t'|}\right\|_{L^\infty(\bm\pi_x)}
\end{align*}
defines a Borel map $X\times (\R^N)^*\to [0,\infty]$ with the following properties.
\begin{itemize}
    \item[(a)] $\Phi_\xi\defeq \Phi(\cdot,\xi)$ is a representative of $|D(\xi\circ\varphi)|_p$ for all $\xi\in (\R^N)^*$;
    \item[(b)] $\Phi^x\defeq \Phi(x,\cdot)$ is a seminorm on $(\R^N)^*$ for $\mu$-a.e. $x\in X$;
    \item[(c)] For any Borel map $\bm\xi\colon U\to(\R^N)^*$ and Borel set $V\subset X$ we have that $\Phi^x(\bm\xi_x)=0$ $\mu$-a.e. $x\in V$ if and only 
    \[
    \bm\xi_{\gamma(t)}((\varphi\circ\gamma)_t')=0\quad a.e.\ t\in \gamma\inv(V)
    \]
    for $p$-a.e. $\gamma$;
    \item[(d)] $(U,\varphi)$ is $p$-independent if and only if $\Phi^x$ is a norm for $\mu$-a.e. $x\in U$.
\end{itemize}
\end{thm}
We will refer to a map $\Phi$  so that (a) in Theorem \ref{thm:canonical_repr_of_grad} holds as a \emph{canonical representative} for the gradient of $\varphi$.

\subsection{From plans to test plans} In the sequel we want to consider canonical representations of gradients of Sobolev functions arising from test plans rather than plans. To achieve this we adapt arguments in \cite[Theorem 8.5 and Theorem 9.4]{amb13}, which we present in detail here for the readers' convenience.

\begin{prop}\label{prop:plans_to_testplans}
	Let $\bm\eta$ be a $q$-plan on $X$, $q\in [1,\infty]$, and $\{ \bm{\pi}_x \}$ the disintegration of $\ud\bm\pi:=|\gamma_t'|\ud t\ud\bm\eta$. Then there exists a $q$-test plan $\bm{\bar\eta}$ on $X$ such that
	\begin{equation}\label{eq:mutual_abs_cont}
		\bm\eta^\#\ll \bm{\bar\eta}^\#\ll \bm\eta^\#
	\end{equation}
	and the disintegration $\{ \bm{\bar\pi}_x \}$ of $\bm{\bar\pi}:=|\gamma_t'|\ud t\ud \bm{\bar\eta}$ satisfies
	\begin{align}\label{eq:invariant_L_infty}
		\left\|\frac{(f\circ\gamma)_t'}{|\gamma_t'|}\right\|_{L^{\infty}(\bm\pi_x)}= \left\|\frac{(f\circ\gamma)_t'}{|\gamma_t'|}\right\|_{L^{\infty}(\bm{\bar\pi}_x)}\quad \bm\eta^\#-a.e.\ x\in X,\quad f\in \Ne pX.
	\end{align}
\end{prop}


\begin{proof}
	First we obtain a plan with parametric barycenter in $L^\infty$ by a suitable parametrization of $\bm\eta$-a.e. curve (see also \cite[Theorem 8.5]{amb13}). Note that if $q=\infty$ the $q$-plan $\bm\eta$ already has parametric barycenter in $L^\infty$ and thus we assume that $q<\infty$.
	
	Let $\bm\eta_1:=\frac{\ell}{1+\ell}\bm\eta_l$, where $\bm\eta_l =l_*\bm\eta$ (with $l\colon C([0,1];X)\to C([0,1];X)$ the constant speed parametrization map). Observe that $\bm\eta^\#\ll (\bm\eta_1)^\#\le\bm\eta^\#$. The \emph{parametric barycenter} $\nu_1:=e_\ast(\ud t\ud\bm\eta_1)$ satisfies $\bm\eta^\#\ll \nu_1\le \bm\eta^\#$ so that $\bm\eta_1$ has parametric barycenter in $L^q$. Denote by $\rho_1$ the density of $\nu_1$ with respect to $\mu$ and set
	\[
	h_1=\frac{1}{\max\{1,\rho_1\}}.
	\]
	For $\gamma\in AC([0,1];X)$, define
	\begin{align*}
		b_\gamma(s):=\int_0^sh_1(\gamma_t)\ud t,\quad \tau_\gamma(s)=\frac{b_\gamma(s)}{b_\gamma(1)},\quad s\in [0,1].
	\end{align*}
	For $\bm\eta_1$-a.e. $\gamma$ we have that $0<b_\gamma(1)\le 1$ and $(\tau_\gamma)'(s)=\frac{h_1(\gamma_s)}{b_\gamma(1)}>0$ a.e. $s\in [0,1]$. For such $\gamma$ it follows that the inverse $\sigma_\gamma\colon [0,1]\to [0,1]$ of $\tau_\gamma$ is absolutely continuous. Moreover, the maps $\gamma\mapsto b_\gamma$ and $H:=\gamma\mapsto \gamma\circ\sigma_\gamma$ are Borel, cf. \cite[Lemma 8.4]{amb13}.
	
	Define $\ud\bm\eta_2:=\frac{b_\gamma(1)}{(1+\ell(\gamma))^q} H_\ast(\ud\bm\eta_1)$. By Lemma \ref{lem:invariance} $\bm\eta_2^\#$ and $\bm\eta_1^\#$ are mutually absolutely continuous since $\frac{b_\gamma(1)}{(1+\ell(\gamma))^q}>0$ $\bm\eta_1$-a.e.. For $\bm\eta_1$-a.e. $\gamma$ and any Borel map $g\colon X\to [0,\infty]$ we have that
	\begin{align*}
		\int_0^1g((\gamma\circ\sigma_\gamma)_t)\ud t=\int_0^1g(\gamma_s)(\tau_\gamma)'(s)\ud s=\frac{1}{b_\gamma(1)}\int_0^1g(\gamma_s)h_1(\gamma_s)\ud s.
	\end{align*}
	Thus, $\nu_2:=e_\ast(\ud t\ud\bm\eta_2)$ 
	satisfies
	\begin{align*}
		\int g\ud\nu_2\le\int \int_0^1g((\gamma\circ\sigma_\gamma)_t)\ud t\ud\bm\eta_1=\int\int_0^1g(\gamma_s)h_1(\gamma_s)\ud s\ud\bm\eta_1=\int gh_1\rho_1\ud\mu.
	\end{align*}
	It follows that $\ud\nu_2\le h_1\rho_1\ud\mu$. In particular $\frac{\ud \nu_2}{\ud\mu}\in L^\infty(\mu)$ and $\nu_1\ll\nu_2\ll\nu_1$. Similarly, for $\bm\eta_1$-a.e. $\gamma$ we have that 
	\begin{align*}
		\int_0^1|(\gamma\circ \sigma_\gamma)'_t|^q\ud t & =\int_0^1 \sigma_\gamma'(t)^q|\gamma'_{\sigma_\gamma(t)}|^q\ud s=\int_0^1\sigma_\gamma'(t)^{q-1}|\gamma_s'|^q\ud s\\
		& =  b_\gamma(1)^{q-1}\int 1\vee\rho_1^{q-1}(\gamma_s)|\gamma_s'|^q\ud s=[b_\gamma(1)\ell(\gamma)]^q\int_0^11\vee\rho_1^{q-1}(\gamma_s)\ud s.
	\end{align*}
	Thus,
	\begin{align*}
		E_q(\bm\eta_2)=&\int\int_0^1|\gamma_t'|^q\ud t\ud\bm\eta_2 = \int\left(\frac{b_\gamma(1)\ell(\gamma)}{1+\ell(\gamma)}\right)^q\int_0^1(1\vee\rho_1)^{q-1}(\gamma_s)\ud s\ud\bm\eta_1(\gamma)\\
		\le & \int\int_0^1\ud s\ud\bm\eta_1+\int\int_0^1\rho_1^{q-1}\ud\ud\bm\eta_1=\bm\eta_1(C([0,1];X))+\int\rho_1^{q}\ud\mu<\infty.
	\end{align*}
	We conclude that $\bm\eta_2$ is a plan with finite $q$-energy and its parametric barycenter $\ud\nu_2=:\rho_2\ud\mu$ satisfies $\rho_2\in L^\infty(\mu)$ and $\bm\eta^\#\ll\nu_2\ll\bm\eta^\#$. A repeated application of Lemma \ref{lem:invariance} implies that
	\begin{equation}\label{eq:2_invariance}
		\|G\|_{L^\infty((\bm\pi_2)_x)}=\|G\|_{L^\infty(\bm\pi_x)}\quad \bm\eta^\#-a.e. \ x\in X
	\end{equation}
	for every $G$ satisfying \eqref{eq:param_invariance}.
	
	We now modify $\bm\eta_2$ to obtain a test plan with the desired properties (cf. \cite[Theorem 9.4]{amb13}). Fix $\eps\in (0,1)$ and given $\tau\in [0,\eps]$, let $r_\tau\colon [0,1]\to [0,1]$, $r_\tau(t)=\tau+(1-\eps)t$. Set
	\begin{align*}
		\bm\eta_\eps:=\frac 1\eps\int_0^\eps r_{\tau\ast}\bm\eta_2\ud\tau.
	\end{align*}
	We claim that $\bm\eta_\eps$ is a test plan. Indeed,
	\begin{align*}
		E_q(\bm\eta_\eps)=\frac 1\eps\int_0^\eps\int\int_0^1|(\gamma\circ r_\tau)'_t|^q\ud t\ud\bm\eta_2\ud\tau=\frac 1\eps\int_0^\eps\int \int_\tau^{1+\tau-\eps}(1-\eps)^{q-1}|\gamma_s'|^q\ud s\ud\bm\eta_2\ud\tau\le E_q(\bm\eta_2),
	\end{align*}
	proving that $\bm\eta_\eps$ has finite $q$-energy. For any Borel $g\colon X\to [0,\infty]$ and $t\in [0,1]$ we may calculate
	\begin{align*}
		\int_X g e_{t\ast}(\ud\bm\eta_\eps) &=\frac 1\eps\int_0^\eps\int g((\gamma\circ r_{\tau})_t)\ud\bm\eta_2\ud \tau=\frac 1\eps \int \int_0^\eps g(\gamma(\tau+(1-\eps)t))\ud\tau\ud\bm\eta_2\\
		&= \frac{1}{\eps}\int\int_{(1-\eps)t}^{\eps+(1-\eps)t}g(\gamma_s)\ud s\ud\bm\eta_2\le \frac{1}{\eps}\int_X g\rho_2\ud\mu.
	\end{align*}
	Thus $\bm\eta_\eps$ is a $q$-test plan. 
	
	It remains to show that $\bm\eta_\eps$ satisfies \eqref{eq:mutual_abs_cont} and \eqref{eq:invariant_L_infty} (for any choice of $\eps>0$). Denote $\ud\bm\pi_\eps=|\gamma_t'|\ud t\ud\bm\eta_\eps$. For any Borel $G\colon C([0,1];X)\times [0,1]\to [0,\infty]$ we have
	\begin{align*}
		\int G\ud\bm\pi_\eps&=\frac 1\eps\int_0^\eps\int \int_0^1G(\gamma\circ r_\tau,t)|(\gamma\circ r_\tau)_t'|\ud t\ud\bm\eta_2\ud\tau\\
		&=\frac 1\eps\int_0^\eps\int\int_\tau^{1+\tau-\eps}G(\gamma\circ r_\tau,r_\tau\inv(s))|\gamma_s'|\ud s\ud\bm\eta_2\ud\tau\\
		&= \int R_\eps G\ud\bm\pi_2,
	\end{align*}
	where
	\begin{align*}
		R_\eps G(\gamma,s):=\frac 1\eps\int_0^\eps\chi_{[\tau,1-\eps+\tau]}(s)G(\gamma\circ r_\tau,r_\tau\inv(s))\ud\tau=\frac 1\eps\int_0^{\eps}\chi_{[s-(1-\eps),s]}(\tau)G(\gamma\circ r_\tau,r_\tau\inv(s))\ud\tau.
	\end{align*}
	If $G$ satisfies \eqref{eq:param_invariance} for all $r_\tau\colon [0,1]\to [0,1]$, $0\le \tau\le \eps$, we obtain
	\begin{align*}
		R_\eps G(\gamma,s)= m_\eps(s)G(\gamma,s), \quad m_\eps(s)=\frac{|[0,\eps]\cap [s-(1-\eps),s]|}{\eps}.
	\end{align*}
	In particular, choosing $G(\gamma,s):=g(\gamma_s)|\gamma_s'|$ for arbitrary Borel $g\colon X\to [0,1]$, we obtain $\ud\bm\eta_\eps^\#=\lambda_\eps\ud\bm\eta_2^\#$ where $\displaystyle \lambda_\eps(x):=\int m_\eps(s)\ud\bm(\pi_2)_x>0$ for $\bm\eta_2^\#$-a.e. $x$. This proves \eqref{eq:mutual_abs_cont}.
	
	The  disintegration $\{ (\bm\pi_\eps)_x \}$ of $\bm\pi_\eps$ satisfies
	\begin{align*}
		\int G\ud(\bm\pi_\eps)_x =\frac{1}{\lambda_\eps(x)}\int R_\eps G\ud(\bm\pi_2)_x\quad \bm\eta^\#-a.e. \ x\in X.
	\end{align*}
	Thus, for $G$ satisfying \eqref{eq:param_invariance} for all $\{r_\tau\colon [0,1]\to [0,1]\}_{\tau\in [0,\eps]}$ we obtain
	\begin{align*}
		\left(\int G^m\ud(\bm\pi_\eps)_x\right)^{1/m} =\left(\frac{1}{\lambda_\eps(x)}\int m_\eps(s) G^m\ud(\bm\pi_2)_x\right)^{1/m} \quad\textrm{for all }m\in \N\quad \bm\eta^\#-a.e.\ x\in X
	\end{align*}
	and letting $m\to \infty$ we obtain
	\begin{align}\label{eq:eps_invariance}
		\|G\|_{L^\infty((\bm\pi_\eps)_x)}=\|G\|_{L^\infty((\bm\pi_2)_x)}\quad \bm\eta^\#-a.e.\ x\in X.
	\end{align}
	
	Finally, given $f\in \Ne pX$ consider $\displaystyle G(\gamma,s)=\frac{(f\circ\gamma)_s'}{|\gamma_s'|}$ and observe that $G$ satisfies \eqref{eq:param_invariance} for every absolutely continuous injection $\sigma\colon [0,1]\to [0,1]$. Now combining \eqref{eq:2_invariance} and \eqref{eq:eps_invariance} we obtain \eqref{eq:invariant_L_infty}. This finishes the proof by choosing e.g. $\bm{\bar\eta}:=\bm\eta_{1/2}$. 	
\end{proof}

\section{Existence of $p$-weak differentials}
We lay some groundwork for the proof of Theorem \ref{thm:equiv_lipdiff_weakchart}. Fix a pair $(U,\varphi)$, and let $\Phi\colon X\times (\R^N)^*\to [0,\infty]$ canonically represent the gradient of $\varphi$. For $\mu$-a.e. $x\in U$ we denote by
\[
|\xi|_x=\Phi(x,\xi),\quad \xi\in (\R^N)^*
\]
the pointwise seminorm, and write $W_x:=(\R^N)^*/\{ \xi:|\xi|_x=0\}$. We also let $L_x\colon W_x\to (\R^N)^*$ denote the right inverse of the canonical projection map $[\cdot]\colon(\R^N)^*\to W_x$, given by sending each $[\xi]\in W_x$ (which is an affine subspace of $(\R^N)^*$) to the unique vector $\zeta\in [\xi]$ with smallest Euclidean norm. Note that $|\cdot|_x$ is a norm on ${\rm Im}(L_x)$.

Define the vector space $\Gamma_p(T^*U)$ as the set of Borel maps $\bm\xi\colon U\to (\R^N)^*$ with $\bm\xi(x)\in {\rm Im}(L_x)$ $\mu$-a.e. such that
\begin{align*}
\|\bm\xi\|_{\Gamma_p(U)}\defeq \left( \int_U|\bm\xi|\ud\mu \right)^{1/p}
\end{align*}
is finite, with the usual identification of elements that agree $\mu$-a.e. It is a standard exercise to show that $\|\cdot\|_{\Gamma_p}$ is a norm making $\Gamma_p(T^*U)$ a Banach space. Moreover we let $\Gamma_{p,loc}(T^*U)$ the collection of all Borel maps $\bm\xi:U\to (\R^N)^*$ so that $\bm\xi|_{K}\in \Gamma_p(T^*K)$ for all compact $K\subset U$. 

\begin{lemma}\label{lem:diff_basic}
Suppose $\bm\xi\colon U\to(\R^N)^*$ is a $p$-weak differential of $f\in \Ne pX$ with respect to $(U,\varphi)$. Then
\begin{itemize}
    \item[(1)] $\bm\xi':=L[\bm\xi]$ is a $p$-weak differential of $f$ with respect to $(U,\varphi)$;
    \item[(2)] We have that $|\bm\xi|=|Df|_p$ $\mu$-a.e. on $U$.
\end{itemize}
\end{lemma}
\begin{proof}
Let $g_f$ be a Borel representative of $|Df|_p$ and define the Borel function $g=\chi_U|\bm\xi|+\chi_{X\setminus U}g_f$. The identity \eqref{eq:differential} implies that
\[
|\bm\xi((\varphi\circ\gamma)_t')|\le g_f(\gamma_t)|\gamma_t'|\quad a.e.\ t\in \gamma\inv(U)
\]
for $p$-a.e. curve. By the definition of $\Phi$ this yields $|\bm\xi(x)|_x\le g_f(x)$ $\mu$-a.e. $x\in U$. Thus $g\le g_f$. The identity \eqref{eq:differential} also yields that for $p$-a.e. curve $\gamma$ we have 
\[
|(f\circ\gamma)'_t|=|\bm\xi_{\gamma(t)}((\varphi\circ\gamma)'_t)|\le g(\gamma_t)|\gamma_t'|
\]
for a.e. $t\in\gamma\inv(U)$. 
Thus $g$ is a $p$-weak upper gradient, and $g\geq g_f$ and (2) follows. To prove (1) note that, since $|\bm\xi'-\bm\xi|=0$ $\mu$-a.e. on $U$ by definition, it follows that $\bm\xi((\varphi\circ\gamma)_t'))=\bm\xi'((\varphi\circ\gamma)_t'))$ a.e. $t\in \gamma\inv(U)$ for $p$-a.e. $\gamma$, cf. \cite[Lemma 4.3(2)]{teriseb}. The claim in (1) follows directly from this.
\end{proof}

\begin{lemma}\label{lem:piecewisegood}
    Let $f\in \Ne pX$. Assume that there exist $f_i\in \Ne pX$, $C_i\subset X$, for each $i\in I$ in a countable index set $I$ such that (1) $f_i$ admits a $p$-weak differential with respect to $(U,\varphi)$ and $f|_{C_i}=f_i|_{C_i}$ for each $i\in I$, and (2) $\displaystyle \mu\left(U\setminus \bigcup_{i\in I}C_i\right)=0$. Then $f$ admits a $p$-weak differential with respect to $(U,\varphi)$. 
\end{lemma}
\begin{proof}
Let $\bm\xi_i$ be a $p$-weak differential of $f_i$ for each $i\in I$. Identify $I$ with a subset of $\N$ and define the sets
\[
W_j\defeq U\cap C_j\setminus \bigcup_{i<j}C_i.
\]
Let $J$ be the set of $j\in I$ for which $\mu(W_j)>0$. Then $\{W_j\}_{j\in J}$ is a partition of $U$ up to a null-set and $f|_{W_j}=f_j|_{W_j}$ for all $j\in J$. We claim that
\[
\bm\xi\defeq\sum_{j\in J}\chi_{W_j}\bm\xi_j\colon U\to (\R^N)^*
\]
is a $p$-weak differential of $f$ with respect to $(U,\varphi)$. Indeed, note that $p$-a.e. curve $\gamma$ in $X$ has the following properties:
\begin{itemize}
    \item[(i)] $f\circ\gamma,\ \varphi\circ\gamma$ and $ f_j\circ\gamma$ are absolutely continuous for all $j\in J$;
    \item[(ii)] almost every $t\in \gamma\inv(U)$ is a density point of $\gamma\inv(W_j)$ for some $j$;
    \item[(iii)] $(f_j\circ\gamma)_t'=(\bm\xi_j)_{\gamma(t)}((\varphi\circ\gamma)_t')$ a.e. $t\in \gamma\inv(U)$.
\end{itemize}
For any such $\gamma$ we have that
\begin{align*}
(f\circ\gamma)_t'=(f_j\circ\gamma)_t'=(\bm\xi_j)_{\gamma(t)}((\varphi\circ\gamma)_t')=\bm\xi_{\gamma(t)}((\varphi\circ\gamma)_t')
\end{align*}
for some $j$, for almost every $t\in \gamma\inv(U)$. This proves the claim. 
\end{proof}
\begin{remark}
 Lemma \ref{lem:piecewisegood} shows in particular the local nature of $p$-weak differentials: if $V\subset U$ and $f_1=f_2$ $\mu$-a.e. on $V$, then a $p$-weak differential $\bm\xi$ of $f_1$ with respect to $(U,\varphi)$ is a $p$-weak differential of $f_2$ with respect to $(V,\varphi)$.
\end{remark}

\begin{lemma}\label{lem:limit_has_diff}
Suppose $(f_j)\subset N^{1,p}_{loc}(X)$ is a sequence such that $f_j\to f$ in $L^p_{loc}(\mu)$ and $(|Df_j|_p)_j$ is equi-integrable. If $f_j$ has a $p$-weak differential with respect to $(U,\varphi)$ for each $j$, then $f$ has a $p$-weak differential with respect to $(U,\varphi)$.
\end{lemma}

\begin{proof}
For each $j$, let $\bm\xi_j$ be a $p$-weak differential of $f_j$. By Lemma \ref{lem:diff_basic} we may assume that $(\bm\xi_j)\subset \Gamma_{p,loc}(T^*U)$.  Since 
\[
\|\bm\xi_j\|_{\Gamma_p(V)}=\left(\int_V|Df_j|_p^p\ud\mu\right)^{1/p},\quad V\subset U,
\]
the sequence $(|\bm\xi_j|^p)$ is equi-integrable. A subsequence of $(\bm\xi_j)$ converges weakly (by reflexivity for $p>1$ and by Dunford--Pettis for $p=1$) to an element $\bm\xi\in \Gamma_{p,loc}(T^*U)$. Denote by $\widetilde{\bm\xi_j}$ and $\widetilde{f_j}$ the sequence of convex combinations (granted by Mazur's lemma) converging to $\bm\xi$ and $f$ in $\Gamma_{p,loc}(T^*U)$ and $L^p_{loc}(\mu)$, respectively, in norm. Now we may argue as in the proof of \cite[Lemma 4.7]{teriseb} (using Fuglede's lemma) to conclude that the identities $(\widetilde{f_j}\circ\gamma)'_t=\widetilde{\bm\xi_j}_{\gamma(t)}((\varphi\circ\gamma)_t')$ for $p$-a.e. $\gamma$ and a.e. $t\in \gamma\inv(U)$ pass to the limit and yield
\[
(f\circ\gamma)_t'=\bm\xi_{\gamma(t)}((\varphi\circ\gamma)_t') \quad a.e.\ t\in \gamma\inv(U)
\]
for $p$-a.e. $\gamma$. 
This proves that $\bm\xi$ is a $p$-weak differential of $f$ with respect to $(U,\varphi)$.
\end{proof}

\begin{remark}
The proof above shows that any weak limit in $\Gamma_p(T^*U)$ of a sequence of $p$-weak differentials of the $f_j$'s is a $p$-weak differential of $f$.
\end{remark}

We close this section by proving Theorem \ref{thm:equiv_lipdiff_weakchart}.
\begin{proof}[Proof of Theorem \ref{thm:equiv_lipdiff_weakchart}]
If $(U,\varphi)$ is a $p$-weak chart, the existence of $p$-weak differentials of Newton--Sobolev functions is proved in \cite[Theorem 1.7]{teriseb}, and their uniqueness follows from $p$-independence and Theorem \ref{thm:canonical_repr_of_grad}(c). The implication (2)$\implies$(3) is trivial. Thus it suffices to prove (3)$\implies$(1).

Assume that every $f\in \LIP(X)$ admits a unique $p$-weak differential with respect to $(U,\varphi)$. It follows that $(U,\varphi)$ is $p$-independent. Indeed, let $\Phi$ represent the gradient of $\varphi$ canonically, cf. Theorem \ref{thm:canonical_repr_of_grad}. Since any Borel map $\bm\xi\colon U\to (\R^N)^*$ with $\bm\xi_x\in \ker \Phi^x$ is a $p$-weak differential of the zero function, the uniqueness of  $p$-weak differentials with respect to $(U,\varphi)$ implies that $\ker \Phi^x=\{0\}$ $\mu$-a.e. $x\in U$. Theorem \ref{thm:canonical_repr_of_grad}(d) implies that $(U,\varphi)$ is $p$-independent.

It remains to show that $(U,\varphi)$ is $p$-maximal. Suppose that $V\subset U$ has positive measure and that $\psi \in {\rm LIP}(X;\R^M)$ is $p$-independent on $V$. We will show that $M\leq N$. Let $\ud \psi_i \in (\R^N)^*$ be the unique $p$-weak differentials of the components $\psi_i$ of $\psi$ for $i=1,\dots M$. To reach a contradiction assume $M>N$. Then $\ud \psi_i$ are linearly dependent, and there are Borel functions $a_i \in L^\infty(V)$ so that $\sum_{i=1}^M a_i \ud \psi_i = 0$, with $\mathbf{a}\defeq (a_1,\dots, a_M) \neq 0$ $\mu$-a.e. on $V$. But, then for $p$-a.e. absolutely continuous $\gamma$ and a.e. $t\in \gamma\inv(V)$ we have 
\[
 \sum_{i=1}^M a_i(\gamma_t) (\psi_i \circ \gamma)'_t=0.
\]
By Theorem \ref{thm:canonical_repr_of_grad}(c) this implies that $\Psi^x(\mathbf{a})=0$ for a.e. $x\in V$, where $\Psi$ canonically represents the gradient of $\psi$. By Theorem \ref{thm:canonical_repr_of_grad}(d) this is a contradiction to $p$-independence.
\end{proof}

\section{Products of charts and tensorization}

\subsection{Tensorization of charts}\label{sec:tensor_chart}

Throughout this section we fix $p$-weak charts $(U,\varphi)$ and $(V,\psi)$ of dimensions $N$ and $M$ in $X$ and $Y$, respectively. To prove Theorem \ref{thm:tensor-charts}, the following two propositions will be used.


\begin{prop}\label{prop:c1-comb_admits_diff}
    Let $u\in N^{1,p}_{loc}(X)$, $v\in N^{1,p}_{loc}(Y)$, $h\in C^1(\R^2)$ and set $f\defeq h\circ(u,v)\in N^{1,p}_{loc}(X\times Y)$. Then the Borel map $\bm\xi\colon U\times V\to (\R^{N+M})^*\simeq(\R^N)^*\times(\R^M)^*$ given by 
    \[
     \bm\xi(x,y)\defeq\partial_1h(u(x),v(y))\ud_xu + \partial_2h(u(x),v(y))\ud_yv
    \]
 is a $p$-weak differential of $f$ with respect to $(U\times V,\varphi\times\psi)$, and 
 \[
     g \defeq \|(|\partial_1h(u(x),v(y))||\ud_xu|, |\partial_2h(u(x),v(y))||\ud_yv|)\|'
    \]
    is a $p$-weak upper gradient of $f$.
\end{prop}
\begin{prop}\label{prop:ptwise_norm_on_prod}
Let $|\cdot|_x$ and $|\cdot|_y$ denote the pointwise norms associated to $\varphi$ and $\psi$, respectively. Then the Borel map $\Xi\colon U\times V\to (\R^{N+M})^*\simeq (\R^N)^*\times(\R^M)^*$ given by
\[
\Xi((x,y),(\xi,\zeta))\defeq \|(|\xi|_x,|\zeta|_y)\|'
\]
canonically represents the gradient of $\varphi\times\psi\in N^{1,p}_{loc}(X\times Y;\R^{N+M})$.
\end{prop}

It follows in particular from Proposition \ref{prop:ptwise_norm_on_prod} that $(U\times V,\varphi\times\psi)$ is $p$-independent. We present the proof of Theorem \ref{thm:tensor-charts} assuming Propositions \ref{prop:c1-comb_admits_diff} and \ref{prop:ptwise_norm_on_prod} above, and after this prove the propositions.

\begin{proof}[Proof of Theorem \ref{thm:tensor-charts}]
Proposition \ref{prop:ptwise_norm_on_prod} implies that $(U\times V,\varphi\times\psi)$ is $p$-independent and thus any $p$-weak differentials are necessarily unique. By Theorem \ref{thm:equiv_lipdiff_weakchart} it suffices to show that every $f\in \LIP(X)$ admits a $p$-weak differential with respect to $(U\times V,\varphi\times,\psi)$. We do this in two steps using Lemma \ref{lem:limit_has_diff}.

We first prove the claim for distance functions. Indeed, let $(x_0,y_0)\in X\times Y$ and $f=d((x_0,y_0),\cdot)$. Note that $f=h\circ(u,v)$, where $h(t,s)=\|(t,s)\|$, $u=d(x_0,\cdot)$ and $v=d(y_0,\cdot)$. Let $h_j\colon \R^2\to [0,\infty]$ be a sequence of smooth functions with uniformly bounded Lipschitz constant converging to $h$ pointwise. By Proposition \ref{prop:c1-comb_admits_diff} the functions $f_j\defeq h_j\circ(u,v)$ admit a $p$-weak differential. The sequence $(f_j)$ is moreover uniformly Lipschitz and thus $(|Df_j|^p)$ is equi-integrable. Since $h_j\to h$ locally uniformly we have that $f_j\to f$ uniformly on bounded sets. By Lemma \ref{lem:limit_has_diff} it follows that $f$ admits a $p$-weak differential, as claimed.

Next we prove the claim for general $f\in \LIP(X\times Y)$ using approximation by a sequence of MacShane extensions. Choose a countable dense set $D\defeq \{(x_1,y_1),(x_2,y_2),(x_3,y_3),\ldots\}\subset X\times Y$ and define
\[
f_N(x)=\min\{ f(x_j)+\LIP(f)d(x,(x_j,y_j)):\ 1\le j \le N\}. 
\]
Note that $\sup_N\LIP(f_N)<\infty$ and $f_N\to f$ pointwise. By Lemma \ref{lem:limit_has_diff} $f$ has a $p$-weak differential if $f_N$  has a $p$-weak differential with respect to $(U\times V,\varphi\times \psi)$ for each $N\in\N$. To see this, observe that there exists a Borel partition $B_1,\ldots, B_N$ of $X$ such that
\[
f_N=f(x_j)+\LIP(f)d((x_j,y_j),\cdot)\textrm{ on }B_j
\]
for each $1\le j\le N$, and that
\[
d((x_j,y_j),\cdot)=\|(d(x_j,\cdot),d(y_j,\cdot))\|.
\]
It follows from Lemma \ref{lem:piecewisegood} that $f_N$ has a $p$-weak differential for each $N$. The claim about the pointwise norm follows from Proposition \ref{prop:ptwise_norm_on_prod}, completing the proof.
\end{proof}

The remainder of this subsection is devoted to proving Propositions \ref{prop:c1-comb_admits_diff} and \ref{prop:ptwise_norm_on_prod}. We start with a technical lemma.
\begin{lemma}\label{lem:diagonal-lemma} Let $h=\tilde h\circ(g_1,g_2)$, where $\tilde h\colon [0,1]^2\to \R$ is Lipschitz and $g_1,g_2\colon [0,1]\to [0,1]$ are absolutely continuous. Assume there exist Borel sets $A,B\subset [0,1]$ such that the following hold: (1) for every $t\in A$ the partial derivative $\partial_1h(t,s)$ exists for every $s\in [0,1]$, and (2) for every $s\in B$ the partial derivative $\partial_2h(t,s)$ exists for every $t\in [0,1]$. Then the function $\delta(t)\defeq h(t,t)$ is absolutely continuous and
\[
\delta'(t)= \partial_1 h(t,t) + \partial_2 h(t,t)\quad a.e.\ t\in A\cap B.
\]
\end{lemma}
\begin{remark} The assumptions can't be much weakened, because if $h(x,y) = \max(|x|,|y|)$, then $\partial_t h(t,t)=1$, but $\partial_1 h,\partial_2 h$ do not exist along the diagonal. In \cite{AGS08} this is avoided by using an upper derivative, but we need the actual derivative to find the differential. The slightly odd assumption on the existence guarantees the identity in the claim.
\end{remark}

\begin{proof}
The absolute continuity of $\delta $ follows from the estimate
\begin{equation}\label{eq:abs_cont}
|\delta(t)-\delta(s)|\le \LIP(\tilde h)[ |g_1(t)-g_1(s)|+|g_2(t)-g_2(s)| ],\quad s,t\in [0,1].
\end{equation}
In particular, the distributional and classical derivatives of $\delta$ agree a.e.. Suppose $\zeta\in C^\infty(\R)$, $\spt\zeta\subset (0,1)$, and $\eps$ is small. On one hand, denoting $C\defeq A\cap B$ we have 
\begin{align*}
	\int_C\zeta(t) \frac{h(t+\eps,t+\eps)-h(t,t)}{\eps} \ud t &= \int_C \zeta(t) \frac{h(t+\eps,t+\eps)-h(t+\eps,t)}{\eps} \ud t\\
	& \quad+  \int_C \zeta(t) \frac{h(t+\eps,t)-f(t,t)}{\eps} \ud t \\
	&=  \int_{C-\eps} \zeta(t-\eps) \frac{h(t,t)-h(t,t-\eps)}{\eps} \ud t \\
	&\quad+  \int_C \zeta(t) \frac{h(t+\eps,t)-h(t,t)}{\eps} \ud t.
\end{align*}
On the other hand, 
\begin{align*}
&\left|\int_{C} \zeta(t) \frac{h(t,t)-h(t,t-\eps)}{\eps} \ud t-\int_{C-\eps} \zeta(t-\eps) \frac{h(t,t)-h(t,t-\eps)}{\eps} \ud t\right|\\
\le & \int_C|\zeta(t-\eps)-\zeta(t)|\frac{|h(t,t)-h(t,t-\eps)|}{\eps} \ud t+ \int_{C\triangle(C-\eps)}\|\zeta\|_\infty \frac{|h(t,t)-h(t,t-\eps)|}{\eps} \ud t
\end{align*}
and the right hand side tends to zero as $\eps\to 0$. Using this,  dominated convergence and the fact that $\partial_1h(t,t),\partial_2h(t,t)$ exist for every $t\in C$ we may take the limit $\eps\to 0$ to obtain
\begin{align*}
\int_C\zeta \delta' \ud t=\int_C\zeta(\partial_1h+\partial_2h)\ud t.
\end{align*}
Since $\zeta$ is arbitrary the claim follows.
\end{proof}

\begin{proof}[Proof of Proposition \ref{prop:c1-comb_admits_diff}]
By Lemma \ref{lem:diff_basic} the functions $|\ud u|$ and $|\ud v|$ are $p$-weak upper gradients for $u$ and $v$ respectively. By using this observation, find curve families $\Gamma_X\subset AC([0,1];X)$ and $\Gamma_Y\subset AC([0,1];Y)$ of zero $p$-modulus such that $u\circ\alpha$ and $v\circ\beta$ are absolutely continuous and
\begin{align*}
(u\circ\alpha)_t'&=\ud_{\alpha(t)}u((\varphi\circ\alpha)_t')\quad a.e.\ t\in \alpha\inv(U)\\
|\ud_{\alpha(t)}u((\varphi\circ\alpha)_t')|&\leq |\ud_{\alpha(t)} u||\alpha_t'|\quad a.e.\ t\in \alpha\inv(U) \\
(v\circ\beta)_t'&=\ud_{\beta(t)}u((\psi\circ\beta)_t')\quad a.e.\ t\in \beta\inv(V) \\
|\ud_{\beta(t)}v((\psi\circ\beta)_t')|&\leq |\ud_{\beta(t)} v||\beta_t'|\quad a.e.\ t\in \beta\inv(V) \\
\end{align*}
whenever $\alpha\notin \Gamma_X$ and $\beta\notin\Gamma_Y$. 

The curve family $\Gamma_0=\Gamma_X\times AC([0,1];Y)\cup AC([0,1];X)\times\Gamma_Y$ has zero $p$-modulus in $X\times Y$. For every $\gamma=(\alpha,\beta)\notin\Gamma_0$ define $H(t,s)=f(\alpha(t),\beta(s))$ and note that $H=h\circ (u\circ\alpha,v\circ\beta)$. Denote $A=\alpha\inv(U)\setminus E$ and $B\defeq \beta\inv(V)\setminus F$ for suitable null-sets $E,F\subset [0,1]$ where the identities above fail for $u$ and $v$, respectively. We have that
\begin{align*}
\partial_1H(t,s)=\partial_1h(u(\alpha_t),v(\beta_s))(u\circ\alpha)_t'=\partial_1h(u(\alpha_t),v(\beta_s))\ud_{\alpha(t)}u((\varphi\circ\alpha)_t')
\end{align*}
exists for every $(t,s)\in A\times [0,1]$, and 
\begin{align*}
\partial_2H(t,s)=\partial_2h(u(\alpha_t),v(\beta_s))(v\circ\beta)_s'=\partial_2h(u(\alpha_t),v(\beta_s))\ud_{\beta(t)}u((\psi\circ\beta)_t')
\end{align*}
exists for every $(t,s)\in [0,1]\times B$. Applying Lemma \ref{lem:diagonal-lemma} we obtain that $f\circ\gamma(t)=H(t,t)$ is absolutely continuous and
\begin{align*}
(f\circ\gamma)_t'&=\partial H_1(t,t)+\partial_2H(t,t)\\
&=\partial_1h(u(\alpha_t),v(\beta_t))\ud_{\alpha(t)}u((\varphi\circ\alpha)_t')\\
&\quad+\partial_2h(u(\alpha_t),v(\beta_t))\ud_{\beta(t)}u((\psi\circ\beta)_t')\\
&=\bm\xi_{(x,y)}(((\varphi\times\psi)\gamma)_t')\quad a.e.\ t\in \alpha\inv(U)\cap \beta\inv(V)=\gamma\inv(U\times V).
\end{align*}
This proves that $\bm\xi$ is a $p$-weak differential of $f$ with respect to $(U\times V,\varphi\times \psi)$.
From the previous equality together with $\|\ud_{\alpha(t)}u((\varphi\circ\alpha)_t')\|\leq \|\ud_{\alpha(t)} u\||\alpha_t'|$ and $\|\ud_{\beta(t)}v((\psi\circ\beta)_t')\|\leq \|\ud_{\beta(t)} v\||\beta_t'|$ we get
\begin{align*}
|(f\circ\gamma)_t'| &=\left|\partial_1h(u(\alpha_t),v(\beta_t))\ud_{\alpha(t)}u((\varphi\circ\alpha)_t')+\partial_2h(u(\alpha_t),v(\beta_t))\ud_{\beta(t)}u((\psi\circ\beta)_t')\right| \\ 
&\leq \left|\partial_1h(u(\alpha_t),v(\beta_t))\right|\left|\ud_{\alpha(t)}u\right|\left|\alpha_t'\right|+ |\partial_2h(u(\alpha_t),v(\beta_t))||\ud_{\beta(t)}v||\beta_t'| \\
&\defeq A \left|\alpha_t'\right|+B|\beta_t'| \leq \|(A,B)\|'\|(\left|\alpha_t'\right|,|\beta_t'|)\| =g(\alpha_t,\beta_t) |(\alpha,\beta)_t'|.
\end{align*}
In particular, $g$ is a $p$-weak upper gradient.
\end{proof}

\begin{proof}[Proof of Proposition \ref{prop:ptwise_norm_on_prod}] We will show that $\Xi((x,y),(\xi,\zeta))$ is a minimal $p$-weak upper gradient for $\xi\circ \varphi + \zeta \circ \psi$ for any $(\xi,\zeta) \in (\R^N \times \R^M)^*$.

By Proposition \ref{prop:c1-comb_admits_diff}, applied to $u=\xi \circ \varphi, v=\zeta \circ \psi$ and $h(u,v)=u+v$, we know that $\Xi((x,y),(\xi,\zeta))$ is a $p$-weak upper gradient for $\xi\circ \varphi + \zeta \circ \psi$. Next, we show that $|D(\xi\circ \varphi + \zeta \circ \psi)|_p\geq \Xi((x,y),(\xi,\zeta))$ for $\mu$-a.e. point $(x,y)$. To show this, it suffices to consider any upper gradient $h\in L^p_{loc}(X\times Y)$ of $\xi\circ \varphi + \zeta \circ \psi$ and to show that 
\begin{equation}\label{eq:suff_to_show}
h\geq \Xi((x,y),(\xi,\zeta))\quad \mu\times \nu-a.e.\ (x,y)\in X\times Y.
\end{equation}

Let $\bm\eta$ and $\bm\eta'$ be $q$-plans on $X$ and $Y$, respectively, and $D\subset X$, $D'\subset Y$ Borel sets such that
\begin{align*}
|\xi|_x=&\chi_D(x)\left\|\frac{\xi((\varphi\circ\alpha)_t')}{|\alpha_t'|}\right\|_{L^\infty(\bm\pi_x)}\\
|\zeta|_y=&\chi_{D'}(y)\left\|\frac{\zeta((\psi\circ\beta)_t')}{|\beta'|}\right\|_{L^\infty(\bm\pi'_y)},
\end{align*}
cf. Theorem \ref{thm:canonical_repr_of_grad}. By Proposition \ref{prop:plans_to_testplans} we may assume that $\bm\eta$ and $\bm\eta'$ are $q$-test plans. 
If $\bm \{\bm\pi_x\},\{\bm\pi'_y\}$ are the disintegrations of $\ud\bm\pi=|\alpha_t'|\ud t\ud\bm\eta$ and $\ud\bm\pi'=|\beta_t'|\ud t\ud\bm\eta'$, then $\{\bm\pi_x\times \bm\pi_y'\}$ is the disintegration of $\bm\pi\times\bm\pi'$ with respect to the map
\[
\tilde e:([0,1]\times AC([0,1];X))\times([0,1]\times AC([0,1];Y))\to X\times Y,\quad \tilde e((t,\alpha),(s,\beta))=(\alpha_t,\beta_s),
\]
and is $\mu\times \nu$-a.e. defined in $D\times D'$. Fix such disintegrations. 

First, for a.e. $(x,y)\not\in D\times D'$ we have $\Xi((x,y),(\xi,\zeta))=0$, and we have the trivial bound $|D(\xi\circ \varphi + \zeta \circ \psi)|_p\geq \Xi((x,y),(\xi,\zeta))=0$. In what follows, we will concentrate on $\mu$-a.e. $(x,y)\in D\times D'$. Fix $(\xi,\zeta)\in (\R^{N+M})^*$ and a upper gradient $h\in L^p_{loc}(\mu\times\nu)$ of $\xi\circ\varphi+\zeta\circ\psi$. Let $\gamma=(\alpha,\beta)$ be a curve for which 
\[
\int_0^1\int_0^1h(\alpha(t),\beta(s))\|(|\alpha_t'|,|\beta_s'|)\|\ud t\ud s<\infty,
\]
(notice that $\bm\eta\times\bm\eta'$-almost every $\gamma$ satisfies this.) A Fubini-type argument yields that for a.e. $(t,s)\in [0,1]^2$
\[
\lim_{\eps\to 0}\frac 1\eps\int_0^\eps h(\tilde\alpha(\tau),\tilde\beta(\tau))|(\tilde\alpha,\tilde\beta)_\tau'|\ud\tau=h(\tilde\alpha(0),\tilde\beta(0))|(\tilde\alpha,\tilde\beta)_0'|,
\]
where
\[
\tilde\alpha(\tau)\defeq\alpha(t+a\tau/|\alpha_t'|),\quad \tilde\beta(\tau)\defeq \beta(s+b\tau/|\beta_s'|)
\]
and $(a,b)$ belongs to a  countable dense set $G\subset \R^2$. Thus, for $(t,s)\in[0,1]^2\setminus N$, where $N$ is a null-set, we have that
\begin{align*}
a\frac{\xi((\varphi\circ\alpha)_t')}{|\alpha_t'|}+b\frac{\zeta((\psi\circ\beta)_s')}{|\beta_s'|}& =\xi((\varphi\circ\tilde\alpha)_0')+\zeta((\psi\circ\tilde\beta)_0')\le h(\tilde\alpha(0),\tilde\beta(0))|(\tilde\alpha,\tilde\beta)_0'|\\
& \le  h(\alpha(t),\beta(s))\|(|a|,|b|)\|
\end{align*}
for $(a,b)\in G$.
By continuity we obtain the estimate for all $(a,b)\in \R^2$. By using Theorem \ref{thm:disintegration}, it follows that for $\mu\times\nu$-a.e. $(x,y)\in D\times D'$ we have
\begin{align}\label{eq:norm_is_min_ug}
    \left\| a\frac{\xi((\varphi\circ\alpha)_t')}{|\alpha_t'|}+b\frac{\zeta((\psi\circ\beta)_s')}{|\beta_s'|}\right\|_{L^\infty(\bm\pi_x\times\bm\pi'_y)}\le h(x,y)\|(|a|,|b|)\|\textrm{ for all }(a,b)\in \R^2.
\end{align}

Let $\sigma_x,\sigma_y\in \{\pm 1\}$  be such that 
\[
\left\|\frac{\xi((\varphi\circ\alpha)_t')}{|\alpha_t'|}\right\|_{L^\infty(\bm\pi_x)}= {\rm esssup}_{\bm\pi_x} \sigma_x\frac{\xi((\varphi\circ\alpha)_t')}{|\alpha_t'|}\ \textrm{ and  }\ \left\|\frac{\zeta((\psi\circ\beta)_s')}{|\beta_s'|}\right\|_{L^\infty(\bm\pi_y')}={\rm esssup}_{\bm\pi_y'} \sigma_y\frac{\zeta((\psi\circ\beta)_s')}{|\beta_s'|},
\]
where ${\rm esssup}_\tau$ is the essential supremum with respect to a measure $\tau$.

For a.e. $(x,y)\in D\times D'$, there exists $a,b \in [0,\infty)^2$ with $\|(a,b)\|=1$ and for which $\|(|\xi|_x,|\zeta|_y)\|'=a|\xi|_x+b|\zeta|_y$. Let $a'=\sigma_x a, b'=\sigma_y b$, and apply \eqref{eq:norm_is_min_ug} to get

\begin{align*}
\|(|\xi|_x,|\zeta|_y)\|'&={\rm esssup}_{\bm\pi_x} a'\frac{\xi((\varphi\circ\alpha)_t')}{|\alpha_t'|}+ {\rm esssup}_{\bm\pi_y'} b'\frac{\zeta((\psi\circ\beta)_s')}{|\beta_s'|} \\
&= {\rm esssup}_{\bm\pi_x\times \bm\pi_y'} \left(a'\frac{\xi((\varphi\circ\alpha)_t')}{|\alpha_t'|}+b'\frac{\zeta((\psi\circ\beta)_s')}{|\beta_s'|}\right)\leq h(x,y).
\end{align*}
This establishes \eqref{eq:suff_to_show} and consequently implies
\begin{align*}
\|(|\xi|_x,|\zeta|_y)\|'\le |D(\xi\circ\varphi+\zeta\circ\psi)|_p(x,y) \quad \mu\times\nu-a.e.\  (x,y),
\end{align*}
completing the proof of the proposition.
\end{proof}

\subsection{Isometric inclusion $W^{1,p}\subset J^{1,p}(X,Y)$}\label{sec:tensor_sob}

With the results of Subsection \ref{sec:tensor_chart} we can prove the isometric inclusion of $N^{1,p}(X \times Y)$ in $J^{1,p}(X,Y)$. 

\begin{proof}[Proof of Theorem \ref{thm:tensor_sob}]
The product $X\times Y$ admits a $p$-weak differentiable structure with charts given by products of charts of $X$ and $Y$, cf. Corollary \ref{cor:prod_has_diff_struct}. Thus every $f \in N^{1,p}(X\times Y)$ has a differential $\ud f$ satisfying $\|\ud f\|=|Df|_p$. We will show that $|\ud f|_{(x,y)}=\|(|\ud_xf^y|_x,|\ud_y f_x|_y)\|'$. Given a $p$-weak chart $(U,\varphi)$ of $X$ and $(V,\psi)$ of $Y$ it suffices to prove the identity for almost every $(x,y) \in U\times V$. 

Since $(U\times V, (\varphi,\psi))$ is a $p$-weak chart, we have a local representation of the differential $\ud_{(x,y)} f=(\mathbf{a}_{(x,y)},\mathbf{b}_{(x,y)})$. 
Since $f \in N^{1,p}(X\times Y)$, for almost every $x \in X$, we have $f_x \in N^{1,p}(Y)$, and for almost every $y \in Y$ we have $f^y \in N^{1,p}(X)$. Thus for $p$-a.e. horizontal curve, i.e. for a.e. $y \in V$ and $p$-a.e. $\gamma \in AC([0,1];X)$ we have that
\[
(f^y\circ \gamma)_t' = \ud_{(\gamma_t,y)}f((\varphi \circ \gamma)_t',0) = \mathbf{a}_{(x,y)}((\varphi \circ \gamma)_t')\quad \textrm{a.e. } t \in \gamma^{-1}(U).
\]

However, since this holds for a.e. fixed $y$, and $p$-a.e. $\gamma\in AC([0,1];X)$, by \cite[Lemma 4.5]{teriseb}, we have $\mathbf{a}_{(x,y)}=\ud_x f^y$ for a.e. $x \in U$. In particular, this means that the map $(x,y) \mapsto \ud_x f^y$ is measurable. Similarly, we get $\mathbf{b}_{(x,y)}=\ud_y f_x$ for a.e. $x \in X$ and a.e. $y \in Y$. We have obtained that
\[
\ud_{(x,y)}f=(\ud_xf^y,\ud_yf_x)\quad\mu\times\nu-a.e.\ (x,y)\in U\times V.
\]
It follows from Lemma \ref{lem:diff_basic}(3) and Proposition \ref{prop:ptwise_norm_on_prod} that 
\begin{align*}
|Df|_p(x,y)=|\ud_{(x,y)}f|_{(x,y)}=\|(|\ud_xf^y|_x,|\ud_yf_x|_y)\|'=\|(|Df^y|_p(x),|Df_x|_p(y))\|'
\end{align*}
for $\mu$-a.e. $(x,y)\in U\times V$. This completes the proof.
\end{proof}

 

\section{Properties of Beppo--Levi functions}

In this section we establish a characterization of the Beppo--Levi space $J^{1,p}$ in terms of Newtonian spaces. Note that the isomorphism $N^{1,p}=W^{1,p}$ does not automatically allow one to replace $W^{1,p}$ with $N^{1,p}$ in the definition of $J^{1,p}(X,Y)$. The main result in this section achieves this by providing a good representative. 

\begin{thm}\label{thm:representativeJ1p}
Let $f\in J^{1,p}(X,Y)$. There exists a representative $\tilde f$ of $f$ such that
\begin{itemize}
    \item[(1)] for $\mu$-a.e. $x\in X$ we have $f_x\in \Ne pY$,
    \item[(2)] for $\nu$-a.e. $y\in Y$ we have $f^y\in \Ne pX$,
    \item[(3)] the map $g(x,y)\defeq \|(|Df^y|_p(x),|Df_x|_p(y))\|'$ is a $p$-weak upper gradient of $f$ along HV-curves, and
    \item[(4)] the minimal $p$-weak upper gradient $\tilde g$ of $\tilde f$ along HV-curves satisfies $g\le c\tilde g$, where $c=\|(c_1,c_2)\|$ and $c_1=\|(1,0)\|$, $c_2=\|(0,1)\|$.
\end{itemize}
\end{thm}
Recall the definition of HV-curves in Definition \ref{def:hv_curve}. In particular one obtains an equivalent definition if in (a1) and (a2) one replaces $W^{1,p}$ by $N^{1,p}$. To prove Theorem \ref{thm:representativeJ1p} we develop a notion of concatenation of plans and apply it to plans concentrated on horizontal and vertical curves, which will henceforth be called horizontal and vertical plans, respectively.

\subsection{Concatenation of plans}
The next definition will allow us to pass from horizontal and vertical plans to plans concentrated on HV-curves (HV plans). Here we give the definition and establish the basic properties of concatenation of plans.

\begin{defn}[Concatenation of plans]
Let $\bm\eta$ and $\bm\eta'$ be plans on $X$ with $e_{1\ast}\bm\eta=e_{0\ast}\bm\eta'=:\nu$. Let $\{\bm\eta_x\}$ and $\{\bm\eta'_x \}$ be the disintegrations of $\bm\eta$ and $\bm\eta'$ with respect to the maps $e_1$ and $e_0$, respectively. For $\nu$-a.e. $x\in X$, set 
\begin{align*}
\bm\eta''_x:=a_\ast(\bm\eta_x\times\bm\eta'_x),
\end{align*}
where $a\colon e_1\inv(x)\times e_0\inv(x)\to X$ is the concatenation map $(\alpha,\beta)\mapsto \alpha\beta$, and define the \emph{concatenation $\bm\eta\ast\bm\eta'$} of $\bm\eta$ and $\bm\eta'$ by 
\begin{equation}\label{eq:concatenation}
\bm\eta\ast\bm\eta':=\int_X \bm\eta''_x\ud\nu(x).
\end{equation}
\end{defn}

\begin{remark}
 Notice that in the definition of concatenation the choice of $\bm\eta_x \times \bm\eta'_x$ is somewhat arbitrary. One could produce new plans by choosing measurably other couplings of $\bm\eta_x$ and $\bm\eta'_x$. In particular, for a given plan $\bm\eta$ the concatenation of the restrictions: $\bm\eta|_{[0,t]} \ast \bm\eta|_{[t,1]}$ does not usually give back the original plan $\bm\eta$.
  
\end{remark}

\begin{lemma}\label{lma:restrict_barycenter}
 Let $\bm\eta$ be a plan on $X$ and $s \in [0,1]$.
 If both $\bm\eta|_{[0,s]}$ and $\bm\eta|_{[s,1]}$
 are $q$-plans, then so is $\bm\eta$. Moreover, we have
 	\begin{align}\label{eq:restrict_barycenter}
	\bm\eta^\#=(\bm\eta|_{[0,s]})^\#+(\bm\eta|_{[s,1]})^\#.
	\end{align}
\end{lemma}
\begin{proof}
 	For any bounded Borel function $g\colon X\to [0,\infty)$ we have 
	\begin{align*}
	\int_X g\ud\bm\eta^\#&=\int\int_0^1g(\gamma_t)|\gamma_t'|\ud t\ud\bm\eta\\
	&=\int\left[\int_0^sg(\gamma_t)|\gamma_t'|\ud t+ \int_s^1g(\gamma_t)|\gamma_t'|\ud t\right]\ud\bm\eta\\
	& = \int\int_0^sg(\gamma_t)|\gamma_t'|\ud t\ud\bm\eta+ \int\int_s^1g(\gamma_t)|\gamma_t'|\ud t\ud\bm\eta\\
	& = \int\int_0^1g(\gamma\circ e_{[0,s]}(t))|(\gamma\circ e_{[0,s]})'(t)|\ud t\ud\bm\eta+ \int\int_0^1g(\gamma\circ e_{[s,1]}(t))|(\gamma\circ e_{[s,1]})'(t)|\ud t\ud\bm\eta\\
	& = \int\int_0^1g(\gamma_t)|\gamma_t'|\ud t\ud\bm\eta|_{[0,s]}
	+\int\int_0^1g(\gamma_t)|\gamma_t'|\ud t\ud\bm\eta|_{[s,1]} \\
	&=\int_Xg\ud(\bm\eta|_{[0,s]})^\#+\int_Xg\ud(\bm\eta|_{[s,1]})^\#,
	\end{align*}
	proving \eqref{eq:restrict_barycenter}. This proves the claim.
\end{proof}

As a corollary we have the following.
\begin{cor}
	The concatenation of two $q$-plans $\bm\eta$ and $\bm\eta'$, whenever defined, is a $q$-plan. Moreover, we have that 
	\begin{align*}
	(\bm\eta\ast\bm\eta')^\#=\bm\eta^\#+(\bm\eta')^\#.
	\end{align*}
\end{cor}

\begin{lemma}\label{lma:restrict_fg}
	Suppose that $\bm\eta$ is a $q$-plan, $s \in [0,1]$ and that $f\in L^p(\mu)$ and $g\in L^p(\mu)$ satisfy the inequality
	\begin{align}\label{eq:sob_test}
	\int|f(\gamma_1)-f(\gamma_0)|\ud\bm\pi(\gamma)\le \int\int_0^1g(\gamma_t)|\gamma_t'|\ud t\ud\bm\pi(\gamma),\quad \bm\pi\in \{\bm\eta|_{[0,s]},\bm\eta|_{[s,1]}\}.
	\end{align}
	 Then \eqref{eq:sob_test} is satisfied for $\bm\pi=\bm\eta$. 
\end{lemma}
\begin{proof}
	By the argument in the proof of Lemma \ref{lma:restrict_barycenter}, we have that
	\begin{align*}
	\int|f(\gamma_1)-f(\gamma_0)|\ud\bm\eta
	& \le \int|f(\gamma_s)-f(\gamma_0)|\ud\bm\eta
	+ \int|f(\gamma_1)-f(\gamma_s)|\ud\bm\eta\\
	& = \int|f(\gamma_1)-f(\gamma_0)|\ud\bm\eta|_{[0,s]}
	+ \int|f(\gamma_1)-f(\gamma_0)|\ud\bm\eta|_{[s,1]}\\
	& \le \int\int_0^1g(\gamma_t)|\gamma_t'|\ud t\ud\bm\eta|_{[0,s]}(\gamma) + \int\int_0^1g(\gamma_t)|\gamma_t'|\ud t\ud\bm\eta|_{[s,1]}(\gamma)\\
	& = \int\int_0^1g(\gamma_t)|\gamma_t'|\ud t\ud\bm\eta(\gamma).
	\end{align*}
	This proves the claim.
\end{proof}
As a corollary we get
\begin{cor}
	Suppose $f\in L^p(\mu)$ and $g\in L^p(\mu)$ satisfy the inequality
	\begin{align*}
	\int|f(\gamma_1)-f(\gamma_0)|\ud\bm\pi(\gamma)\le \int\int_0^1g(\gamma_t)|\gamma_t'|\ud t\ud\bm\pi(\gamma),\quad \bm\pi\in \{\bm\eta,\bm\eta'\},
	\end{align*}
	for two $q$-plans $\bm\eta$ and $\bm\eta'$ on $X$ for which $e_{1\ast}\bm\eta=e_{0\ast}\bm\eta'$. Then \eqref{eq:sob_test} is satisfied for $\bm\pi=\bm\eta\ast\bm\eta'$. 
\end{cor}

\subsection{Concatenation of horizontal and vertical curves}

We now apply the lemmas from the previous subsection to HV-plans. Recall that $HV([0,1];X\times Y)$ denotes the set of all HV curves, and $HV_n([0,1;X\times Y)$ the subset of $HV([0,1];X\times Y)$ consisting of HV curves with exactly $n$ turning times. We remark that $HV_0([0,1];X\times Y)$ consists of the union of all horizontal and vertical curves. Our first step is to decompose plans concentrated on $HV_n([0,1];X\times Y)$. For the proof we denote by $l\colon AC([0,1];X\times Y)\to \LIP([0,1];X\times Y)$ the map sending $\gamma$ to its constant speed parametrization $\bar\gamma$.

\begin{lemma}\label{lem:ug_along_hor_and_vert_plans}
Let $\bm\eta$ be a horizontal or vertical plan in $X\times Y$. Let $f\in J^{1,p}(X,Y)$ and $g(x,y)=\|(|Df^y|_p(x),|Df_x|_p(y))\|'$. Then
\[
\int|f(\gamma_1)-f(\gamma_0)|\ud\bm\eta\le \int\int_0^1g(\gamma_t)|\gamma_t'|\ud t\ud\bm\eta.
\]
\end{lemma}
\begin{proof}
We prove the claim for horizontal plans. The vertical case follows similarly. If we replace $\bm\eta$ by $l_\ast\bm\eta$, both sides of the inequality remaining unchanged. Thus we may assume that $\bm\eta$-a.e. $\gamma$ is constant speed parametrized. Further we may assume that $\bm\eta$-a.e. $\gamma$ is non-constant.

Since $\bm\eta$ is a horizontal plan, for $\bm\eta$-a.e. $\gamma=(\alpha,\beta)$ we have that $\beta$ is a constant $y_\beta$. Denote by $h$ the map $\gamma\mapsto y_\beta$ and let $\{\bm\eta_y\}$ be the disintegration of $\bm\eta$ with respect to $h$. Set $\nu\defeq h_\ast\bm\eta$. Observe that $h\inv(y)=AC([0,1];X\times \{y\})\simeq AC([0,1];X)$ and we regard $\bm\eta_y$ as a plan living on the metric space $X\times\{y\}$ whose metric is a constant multiple of the metric of $X$ (constant independent of $y$).

It is not difficult to see that $\nu\ll\mu_Y$. We claim that $\bm\eta_y^\#=\rho(\cdot,y)\mu_X$ for $\nu$-a.e. $y$, where $\rho\in L^q(\mu_X\times \mu_Y)$ is the density of $\bm\eta^\#$ with respect to $\mu_X\times \mu_Y$. It follows from this that $\bm\eta_y$ is a $q$-plan for $\nu$-a.e. $y$.

For each Borel $E\subset X$ and bounded Borel function $g\colon Y\to [0,\infty]$ we have 
\begin{align*}
\int_Yg(y)\bm\eta_y^\#(E)\ud\nu(y)=& \int_Yg(y)\int\int_0^1\chi_E(\alpha_t)|(\alpha,y)'_t|\ud t\ud\bm\eta_y\ud\nu(y) \\
=&\int_Yg(y)\int_{h\inv(y)}\int_0^1\chi_{E\times Y}(\gamma_t)|\gamma_t'|\ud\bm\eta_y\ud\nu(y)=\int g\circ h(\gamma)\int_0^1\chi_{E\times Y}(\gamma_t)|\gamma_t'|\ud t\ud\bm\eta\\
=&\int_{E\times Y}g(y)\ud\bm\eta^\#=\int_Yg(y)\left(\int_E\rho(\cdot,y)\ud\mu_Y\right)\ud\mu_Y(y).
\end{align*}
Since $E$ and $g$ are arbitrary it follows that for $\nu$-a.e. $y$ the identity $\bm\eta_y^\#=\rho(\cdot,y)\mu_X$ holds. 

Now we have the estimate
\begin{align*}
\int|f(\gamma_1)-f(\gamma_0)|\ud\bm\eta=&\int_Y\int_{h\inv(y)}|f^y(\alpha_1)-f^y(\alpha_0)|\ud\bm\eta_y\ud\nu(y)\\
\le &\int_Y\int\int_0^1|Df^y|_p(\alpha_t)|(\alpha,y)_t'|\ud t\ud\bm\eta_y\ud\nu(y)=\int\int_0^1|Df^{\beta(t)}|_p(\alpha_t)|(\alpha,\beta)_t'|\ud t\ud\bm\eta\\
\le & \int\int_0^1\|(|Df^{\beta(t)}|_p(\alpha_t),|Df_{\alpha(t)}|_p(\beta_t))\|'\|(|\alpha_t'|,|\beta_t'|)\|\ud t\ud\bm\eta\\
=&\int\int_0^1g(\gamma_t)|\gamma_t'|\ud t\ud\bm\eta,
\end{align*}
which proves the claim.
\end{proof}

Notice that  a plan $\bm\eta$ concentrated on $HV_0([0,1]:X\times Y)$ has a decomposition $\bm\eta=\bm\eta_H+\bm\eta_V$, where  $\bm\eta_H$ is a horizontal and $\bm\eta_V$ a vertical plan. This can be shown e.g. by disintegrating $\bm\eta$ with respect to the map $P\colon HV_0\to \{0,1\}$ which sends horizontal curves to 0 and vertical curves to 1 (we may by convention send constant curves to 0). Lemma \ref{lem:ug_along_hor_and_vert_plans} directly implies that, for $f\in J^{1,p}(X,y)$ and $g$ as in the claim, the inequality holds for all $q$-plans concentrated on $HV_0$.

Next we extend this observation to plans concentrated on $HV_n$. For the proof we say that two plans $\bm\eta$ and $\bm\eta'$ are \emph{equivalent} if $l_\ast\bm\eta=l_\ast\bm\eta'$. Note that $(l_\ast\bm\eta)^\#=\bm\eta^\#$ and that $l(HV_n([0,1];X\times Y))=HV_n([0,1];X\times Y)$ for all $n$.

\begin{prop}\label{prop:q-planHV}
	Let $f\in J^{1,p}(X,Y)$ and let $g(x,y)=\|(|Df_x|_p(y),|Df^y|_p(x))\|'$. Then, for any $q$-plan $\bm\eta$ concentrated on $HV([0,1];X\times Y)$ we have that
	\begin{align}\label{eq:plan_upper_grad}
	\int|f(\gamma_1)-f(\gamma_0)|\ud\bm\eta\le\int\int_0^1g(\gamma_t)|\gamma_t'|\ud t\ud\bm\eta.
	\end{align}
\end{prop}
\begin{proof}
 Let $\bm\eta$ be a $q$-plan concentrated on $HV([0,1];X\times Y)$. 
 By disintegrating $\bm\eta$ with respect to the map
\[
tn\colon HV([0,1];X\times Y)\to \N,\quad \gamma\mapsto \#\{\textrm{number of turning points of }\gamma\}
\]
we obtain a representation
\begin{align*}
\bm\eta=\lambda_0\bm\eta_0+\lambda_1\bm\eta_1+\lambda_2\bm\eta_2+\cdots,
\end{align*}
where $(\lambda_n)$ is a non-negative summable sequence and, for each $n\in\N$ for which $\lambda_n>0$, $\bm\eta_n$ is a $q$-plan concentrated on $HV_n([0,1];X\times Y)$. 
Thus, we may assume that $\bm\eta$ is concentrated on $HV_n([0,1];X\times Y)$ for some $n \in \N$.

Since both sides of \eqref{eq:plan_upper_grad} are invariant under replacing $\bm\eta$ by $l_\ast\bm\eta$ we may moreover assume that $\bm\eta$-a.e. $\gamma$ is constant speed parametrized. For each such $\gamma$, let $\bm t(\gamma)\in (0,1)^n$ denote the turning times of $\gamma$, (cf. Definition \ref{def:hvn_curve}). Further, let $r_\gamma\colon [0,1]\to [0,1]$ be the piecewise affine bijection for which $r_\gamma(\bm t(\gamma)_i)=i/n$ for each $i=1,\ldots,n$, and set $r(\gamma)\defeq \gamma\circ r_\gamma$. The plan $\bm\eta'\defeq r_\ast\bm\eta$ is equivalent to $\bm\eta$ and for each $i=1,\ldots,n$ the restriction $\bm\eta_i=\bm\eta'|_{[(i-1)/n,i/n]}$ is concentrated on $HV_0([0,1];X\times Y)$. Thus \eqref{eq:plan_upper_grad} holds for $\bm\eta_i$ for each $i=1,\ldots,n$. By (an iterated use of) Lemma \ref{lma:restrict_fg} it follows that \eqref{lma:restrict_fg} holds for $\bm\eta'$, and thus for $\bm\eta$. This completes the proof.
\end{proof}

Using arguments from \cite{amb13} and \cite{teriseb}, Proposition \ref{prop:q-planHV} yields the following corollary.
\begin{cor}\label{cor:representativeJ1p}
For each $f \in J^{1,p}(X \times Y)$ there exists a $\tilde f \in J^{1,p}(X \times Y)$ so that $f = \tilde f$ a.e. and
 $g(x,y)=\|(|Df_x|_p(y),|Df^y|_p(x))\|'$ is an upper gradient of $\tilde f$ on $p$-a.e. hv-curve.
\end{cor}
\begin{proof}
Arguing as in the proof of \cite[Lemma 3.3]{teriseb}, Proposition \ref{prop:q-planHV} implies that for $\Mod_p$-a.e. $HV$-curve $\gamma$ there exists a representative $\tilde f_\gamma$ of $f\circ\gamma$ for which the function $g$ is an upper gradient of $\tilde f_\gamma$ along $\gamma$.
Now repeating the argument in the proof of \cite[Theorem 10.3]{amb13} we obtain the required representative $\tilde f$. 
\end{proof}

\begin{proof}[Proof of Theorem \ref{thm:representativeJ1p}]
Let $f\in J^{1,p}(X,Y)$. The representative $\tilde f$ of $f$ given by Corollary \ref{cor:representativeJ1p} satisfies (1) and (2), since $g^y$ and $g_x$ are $p$-weak upper gradients of $f^y$ and $f_x$, respectively, for $\mu_X$-a.e. $x\in X$ and $\mu_Y$-a.e. $y\in Y$. If $h$ is a $p$-weak upper gradient of $\tilde f$ along HV-curves, then for $\mu$-a.e. $(x,y)$ we have
\begin{align*}
|(\tilde f^y\circ\alpha)_t'|\le h(\alpha_t,y)|\alpha_t'|\|(1,0)\|,\quad |(\tilde f_x\circ\beta)_t'|\le h(x,\beta_t)|\beta_t'|\|(0,1)\|\ \text{a.e. } t
\end{align*}
for $p$-a.e. curves $\alpha$ and $\beta$ in $X$ and $Y$, respectively. Consequently $|Df^y|_p(x)\le c_1h(x,y)$ and $|Df_x|_p(y)\le c_2h(x,y)$ and thus
\begin{align*}
g(x,y)=\sup\{ a|Df^y|_p(x)+b|Df_x|_p(y):\ \|(a,b)\|=1,\ a,b\ge 0 \}\le ch(x,y)\quad \mu-\text{a.e. }(x,y)\in X\times Y
\end{align*}
proving (4).
\end{proof}

\section{Tensorization when one factor is a PI-space}
In this section we prove Theorem \ref{thm:Onefactor}. Our proof uses a characterization of the tensorization property in terms of (weak) approximation properties. Specifically, we use the doubling property and Poincar\'e inequality to construct Lipschitz approximants of a Beppo--Levi function satisfying (3) in Proposition \ref{prop:equivalentconditions} below.


\begin{prop}\label{prop:equivalentconditions}Assume $p \in (1,\infty)$. Suppose that $X$ and $Y$ admit $p$-weak differentiable structures. If any of the following (equivalent) conditions is satisfied, then $W^{1,p}(X\times Y)=J^{1,p}(X,Y)$.

\begin{enumerate}
\item \textbf{Density:} $W^{1,p}(X\times Y)$ is dense in $J^{1,p}$ in norm.
\item \textbf{(Weak) Approximation by Sobolev functions:} For every bounded $f \in J^{1,p}(X,Y)$ with bounded support, there exists a constant $C>0$ and a sequence of functions $f_i \in N^{1,p}(X\times Y)$ with $\|f_i\|_{N^{1,p}} \leq C$ so that $f_i$ converges to $f$ in $L^p(X\times Y)$.
\item \textbf{(Weak) Approximation by Lipschitz functions:} For every bounded $f \in J^{1,p}(X,Y)$ with bounded support, there is a constant $C>0$ and a sequence of Lipschitz functions $f_i \in J^{1,p}(X,Y)$ with bounded support, so that $\|f_i\|_{W^{1,p}} \leq C$, and so that $f_i$ converges to $f$ in $L^{p}(X\times Y)$.
\end{enumerate}
\end{prop}

\begin{proof}

Since $W^{1,p}(X\times Y) \subset J^{1,p}(X,Y)$ is an  isometric closed subset, then equality $W^{1,p}(X\times Y)=J^{1,p}(X,Y)$ is equivalent to $(1)$.

Clearly $(1) \Longrightarrow (2)$. Further, $(2) \Longleftrightarrow (3)$ by density in Energy of Lipschitz functions, for $p \in (1,\infty)$ see \cite{ambgigsav} (alternative proof in \cite{seb2020}).

Next, we show that $(3)$ implies $(1)$. Suppose $f \in J^{1,p}(X, Y)$.

By Corollary \ref{cor:prod_has_diff_struct}, $X \times Y$ admits a $p$-weak differential structure. By \cite[Corollary 6.7]{teriseb}, the space $W^{1,p}(X \times Y)$ is reflexive. Thus bounded functions with bounded support are dense in $W^{1,p}(X\times Y)$.

Let $f\in J^{1,p}(X,Y)$ be bounded and with bounded support. Let $(f_i)$ be a sequence of Lipschitz functions satisfying the conclusion in (3). By reflexivity and Mazur's Lemma, a convex combination $\tilde f_i$ of a suitable subsequence of  $(f_i)$ converges in the norm of $W^{1,p}(X\times Y)$ and in $L^{p}(X\times Y)$. Since the inclusion $W^{1,p}(X\times Y) \hookrightarrow L^p(X\times Y)$ is injective, and since the sequence also converges in $L^p(X\times Y)$ to $f$, then the limit in $W^{1,p}(X\times Y)$ equals $f$. Therefore, $f\in W^{1,p}(X\times Y)$. The density of $W^{1,p}(X\times Y)$ in $J^{1,p}(X,Y)$ follows.
\end{proof}



To construct Lipschitz approximants in the proof of Theorem \ref{thm:Onefactor} we use a so-called discrete convolution in the $X$-direction. First, define a Lipschitz partition on unity. For $n \in \N$ fix an $2^{-n}$-net $N_n \subset X$. That is, fix a set $N_n \subset X$ so that for each $x \in X$ there is a $a \in N_n$
 with $d(a,x) \leq 2^{-n}$, and for each $a,b \in N_n$ we have $d(a,b) > 2^{-n}$.  Let $\{\psi^n_a : a \in N_n\}$ be a Lipschitz partition of unity subordinate to the cover $\{B(a,2^{1-n})\}$ so that ${\rm supp}(\psi^n_a) \subset B(a,2^{2-n})$. The functions $\psi^n_a$ can be chosen to be $C(D)2^{n}$--Lipschitz with a constant $C(D)$ depending on doubling of $X$. We will fix this partition of unity and the nets $N_n$ in the proofs below.

For a function $f\in W^{1,p}(X)$ define the approximation 
 
 \begin{equation}
 T_n f(x) = \sum_{a \in N_n} \psi^n_a(x) \vint_{B(a,2^{-n})} fd\mu.
 \end{equation}
 
We have the following lemma. 

\begin{lemma}\label{lem:discrete-convolution}
If $X$ is $p$-PI, then $T_n:W^{1,p}(X)\to W^{1,p}(X)$ is bounded, with norm bounded independent of $n$, and $T_n f \to f$ in $L^p(X)$ for every $f\in W^{1,p}(X)$.
\end{lemma}
\begin{proof} That $T_n:L^{p}(X)\to L^p(X)$ is bounded, and that $T_n f \to f$ for all $f \in L^p(X)$ follows from \cite[Lemma 5.2]{HKTapprox}. Further, \cite[Lemma 5.2]{HKTapprox} also implies that there is a constant $C>0$ independent of $n$ so that  $T_n f$ is locally Lipschitz with
\[
\Lip[T_nf](x) \leq C 2^{n}\dashint_{B(a,5\lambda 2^{2-n})} |f-f_{B(a,2^{-n})}| d\mu
\]
whenever $x\in B(a,5\lambda 2^{2-n})$. By the $p$-Poincar\'e inequality, we have
\[
|DT_nf|_p(x)^p \leq \Lip[T_nf](x)^p \leq C^p c_{PI}^p \dashint_{B(a,5\lambda 2^{2-n})} |Df|_p^p d\mu
\]
for all $x\in B(a,5\lambda 2^{2-n})$. The balls $B(a,5\lambda 2^{2-n})$ have  bounded overlap by the doubling condition and thus we get

\[
\int |DT_nf|_p(x)^p d\mu \leq \sum_{a \in N_n} C^pc_{PI}^p \int_{B(a,5\lambda 2^{2-n})} |Df|_p^p d\mu \leq C' \int_{X} |Df|_p^p d\mu,
\]
for a constant $C'=C(C,D,\lambda, c_{PI})$, where $D$ is the doubling constant of $\mu$.
\end{proof}

\begin{remark}
Indeed, our proof of Theorem \ref{thm:Onefactor} will use the Poincar\'e inequality only through the previous Lemma. One may conjecture that linear approximating operators $T_n$ exist more generally. The crucial properties that we need are that the expression is given by a partition of unity, averages of the function, and that it is a bounded linear operator. PI-spaces are the only context where such a discrete convolution operators are known to exist. One plausible approach to tensorization of Sobolev spaces would involve constructing such convolutions more generally.
\end{remark}

\begin{proof}[Proof of Theorem \ref{thm:Onefactor}]
We will verify $(3)$ from Proposition \ref{prop:equivalentconditions}. Let $f \in J^{1,p}(X,Y)$ be bounded and with bounded support. We assume that $f$ is the good representative given by Theorem \ref{thm:representativeJ1p}. 

For each $a \in N_n$ consider the function $f^a \colon  y \mapsto \dashint_{B(a,2^{-n})} f(z,y) d\mu(z).$ Let $g(z,y)=|Df_z|_p(y)$ be the $p$-weak upper gradient of $f$ in the $Y$-direction and define $\tilde{g}^a(y) = \dashint_{B(a,2^{-n})}|D f_z|_p(z,y)d\mu(z).$ By Minkowski's inequality, we see that $\tilde{g}^a \in L^p$ with
\[
\|\tilde{g}^a\|_{L^p(Y)} \leq \dashint_{B(a,2^{-n})} \|g(z, \cdot)\|_{L^p(Y)} d\mu(z).
\]

 Let $\bm \eta$ be any $q$-test plan in $Y$.

\begin{align*}
\int |f^a(\gamma_1)-f^a(\gamma_0)|d\bm \eta(\gamma) &\leq \dashint_{B(a,2^{-n})} \int |f(z,\gamma_1)-f(z,\gamma_0)|d\bm \eta(\gamma) d\mu(z)\\
&\leq \dashint_{B(a,2^{-n})} \int\int_0^1 |Df_z|_p(z,\gamma_t) |\gamma_t'| dt d\bm \eta(\gamma) d\mu(z)\\
&\leq  \int\int_0^1 \dashint_{B(a,2^{-n})}|Df_z|_p(z,\gamma_t)d\mu(z) |\gamma_t'| dt d\bm \eta(\gamma)\\
&\leq \int\int_\gamma \tilde{g}^a ds d\bm \eta(\gamma).
\end{align*}
Since $\bm \eta$ is arbitrary it follows that $f^a\in W^{1,p}(Y)$.

Define 
\[
f_n(x,y) =  T_n(f^y)(x) = \sum_{a \in N_n} \psi^n_a(x) f^a(y).
\]
 
Since $f$ has bounded support, the sum here is in fact finite. It follows that $f_n\in N^{1,p}(X\times Y)$.  
The doubling condition and $p$-Poincar\'e inequality on $X$ easily implies that $f_n \to f$ in $L^p(X \times Y)$. Indeed, 
\begin{align*}
	f(x,y)-f_n(x,y)=\sum_{a\in N_n}\psi^n_a(x)[f(x,y)-f^a(y)]=\sum_{a\in N_n}\psi^n_a(x)[f^y(x)-(f^y)_{B(a,2^{-n})}].
\end{align*}
whence 
\begin{align*}
	\left(\int_X|f_n(x,y)-f(x,y)|^p\ud\mu(x)\right)^{1/p}&\le \sum_{a\in N_n}\left(\int_{B(a,2^{2-n})}|f^y-(f^y)_{B(a,2^-n)}|^p\ud\mu\right)^{1/p}\\
	&\le C\sum_{a\in N_n}2^{-n}\left(\int_{B(a,2^{2-n})}|Df^y|^p\ud\mu\right)^{1/p}.
\end{align*}
Integrating the $p$-th power of this estimate over $Y$ using that the balls $B(a,2^{2-n})$ have bounded overlap we obtain $\|f_n-f\|_{L^p(X\times Y)}\le C2^{-n}\|f\|_{J^{1,p}}$, implying $\|f_n-f\|\to 0$ as claimed.

Now, we show that $\sup_{n\in\N}\||Df_n|_p\|_{L^p}<\infty$. We have $|Df_n|_p\leq |\ud_X f_n| + |\ud_Y f_n|$. First consider $|\ud_Y f_n|$. Since $x$ is constant we obtain

\[
\ud_Y f_n(x,y) = \sum_{a \in N_n} \psi^n_a(x) \ud_Y f^a(y).
\]

By the triangle inequality, and since the sum has at most $C(D)$-nonzero terms for any given $x$, we get
\begin{align*}
\int_Y |\ud_Y f_n|^p(x,y)d\nu(y) \lesssim_{D} & \sum_{a \in N_n} |\psi^n_a(x)|^p \int_Y |\ud_Y f^a(y)|^p d\nu(y) \\
\lesssim & \sum_{a \in N_n} |\psi^n_a(x)|^p  \dashint_{B(a,2^{-n})} \|g(z, \cdot)\|_{L^p(Y)} d\mu(z). 
\end{align*}
 
 Integrating over $X$, noting that $\psi^n_a$ has support in $B(a,2^{2-n})$, and using the doubling condition yields that $\||\ud_Y f_n|\|_{L^p(X\times Y)} \lesssim |g|_{L^p(X\times Y)}$. 
 
 Next, consider the $X$-derivative. Let $b\in N_n$. We may write
 \[
 \ud_Xf_n(x,y)=\sum_{a\in N_n}\ud_X\psi^n_a(x)[f^a(y)-f^b(y)]
 \]
using the linearity of the differential and the fact that $\psi^a_n$ is a partition of unity. The doubling condition implies the existence of a constant $C_D$ such that, for each $x\in B(b,2^{2-n})$, there are at most $C_D$ elements $a\in N_n$ with $B(a,2^{2-n})\cap B(b,2^{2-n})\ne \varnothing$. For these $a\in N_n$ the $p$-Poincar\'e inequality implies the estimate
\begin{align*}
	|f^a(y)-f^b(y)|=\left|\dashint_{B(a,2^{-n})}f(\cdot,y)\ud\mu-\dashint_{B(b,2^{-n})}f(\cdot,y)\ud\mu\right| \le C2^{-n}\left(\dashint_{B(b,\lambda 2^{8-n})}|Df^y|^p\ud\mu\right)^{1/p}.
\end{align*}
This implies that
\begin{align*}
\|\ud_Xf_n(x,\cdot)\|_{L^p(Y)}&\le \sum_{a\in N_n}|\ud_X\psi_a^n(x)|\|f^a(y)-f^b(y)\|_{L^p(Y)} \\
&\le C\sum_{a\in N_n}2^n\cdot C2^{-n}\left\|\left( \dashint_{B(b,\lambda 2^{8-n})}|Df^y|^p\ud\mu\right)^{1/p} \right\|_{L^p(Y)}\\
&\le C'\left( \dashint_{B(b,\lambda 2^{8-n})}\||Df^y|(z)\|_{L^p(Y)}^p\ud\mu(z)\right)^{1/p}
\end{align*}
for each $x\in B(b,2^{2-n})$. The balls $B(b,\lambda 2^{8-n})$ have bounded overlap by the doubling condition. Thus

\begin{align*}
	\int_{X\times Y}|\ud_Xf_n|^p\ud(\mu\times\nu)&=\int_X\||\ud_Xf_n(x,\cdot)\|_{L^p(Y)}^p\ud\mu(x)=\sum_{b\in N_n}\int_{X}\psi^n_b(x)\||\ud_Xf_n(x,\cdot)\|_{L^p(Y)}^p\ud\mu(x)\\
	&\le C'\sum_{b\in N_n}\int_{B(b,2^{2-n})}\dashint_{B(b,\lambda 2^{8-n})}\||Df^y|(z)\|_{L^p(Y)}\ud\mu(z)\ud\mu(x)\\
	&\le C'\sum_{b\in N_n}\int_{B(b,\lambda 2^{8-n})}\||Df^y|(z)\|_{L^p(Y)}^p\ud\mu(z)\le C''\int_X\||Df^y|(z)\|_{L^p(Y)}^p\ud\mu(z).
\end{align*}
Thus $\|\ud_Xf_n\|_{L^p(X\times Y)}$ and $\|\ud_Yf_n\|_{L^p(X\times Y)}$ are bounded independently of $n$, completing the proof.
 \end{proof}

\appendix

\section{}
\subsection{Elementary properties of disintegration}
We record some elementary properties of disintegrations of plans. In the following statement, the space $ACB([0,1])$ consists of absolutely continuous bijections $\sigma\colon [0,1]\to [0,1]$ with absolutely continuous inverse, and $\{\bm\pi_x\}$ is the disintegration of the measure $\bm\pi$ given by $\ud\bm\pi:=|\gamma_t'|\ud t\ud \bm\eta$ with respect to $e$. The following properties are easy to verify from the definition by direct calculation using the uniqueness of disintegration. Thus we omit the proofs.

\begin{lemma}\label{lem:basic_disint}
	Let $F\colon AC([0,1];X)\to [0,\infty]$ be bounded and Borel and denote
	\[
	f(x):=\int_{e\inv(x)}F\ud\bm\pi_x
	\]
	for $\bm\eta^\#$-a.e. $x\in X$. Let $\sigma\colon C([0,1];X)\to ACB([0,1])$ be Borel and denote 
    \[
    H\colon C([0,1];X)\to C([0,1];X),\quad \gamma\mapsto \gamma\circ \sigma_\gamma.
	\]
	\begin{itemize}
		\item [(a)] The plan $\bm\eta_F:=F\bm\eta$ and the disintegration $\{(\bm\pi_F)_x\}$ of $\ud\bm\pi_F:=|\gamma_t'|\ud t\ud\bm\eta_F$ satisfy
		\[
		\ud\bm\eta_F^\#(x)=f(x)\ud \bm\eta^\#(x),\quad (\bm\pi_F)_x=\frac{1}{f(x)} F\bm\pi_x\quad \bm\eta^\#-a.e.\ x\in \{f>0\}.
		\]
		\item[(b)] The plan  $\bm\eta_H:=H_\ast\bm\eta$ and the disintegration $\{(\bm\pi_H)_x\}$ of $\ud\bm\pi_H:=|\gamma_t'|\ud t\ud\bm\eta_H$ satisfy
		\[
		\bm\eta_H^\#=\bm\eta^\#,\quad (\bm\pi_H)_x=(H_{\sigma\inv})_\ast\bm\pi_x \quad \bm\eta^\#-a.e.\ x\in X.
		\]
		
		Here $H_{\sigma\inv}(\gamma,t):=(H\gamma,\sigma_\gamma\inv(t))$.
	\end{itemize}
\end{lemma}

\begin{lemma}\label{lem:invariance}
	Let $F$ and $H$ be as in Lemma \ref{lem:basic_disint} and suppose that $F>0$ $\bm\eta$-a.e. Consider $\bm\eta_{F,H}:=FH_\ast\bm\eta$ and the disintegration $\{ (\bm\pi_{F,H})_x \}$ of $\ud\bm\pi_{F,H}:=|\gamma_t'|\ud t\ud \bm\eta_{F,H}$. Then $\bm\eta^\#\ll \bm{\eta}_{F,H}^\#\ll \bm\eta^\#$ and
	\begin{align}\label{eq:sup_norm_invariance}
	&\|G\|_{L^\infty((\bm\pi_{F,H})_x)}=\|G\|_{L^\infty(\bm\pi_x)}\quad \bm\eta^\#-a.e.\ x\in X
	\end{align}
	for every Borel function $G\colon C([0,1];X)\times[0,1]\to \R$ satisfying
	\begin{equation}\label{eq:param_invariance}
	G(\gamma\circ\sigma_\gamma,t)=G(\gamma,\sigma_\gamma(t))\quad \bm\pi-a.e.\ (\gamma,t).
	\end{equation}
\end{lemma}
\begin{proof}
	Lemma \ref{lem:basic_disint} implies that $\bm\eta^\#\ll \bm{\eta}_{F,H}^\#\ll \bm\eta^\#$ (since $f>0$ $\bm\eta^\#$-a.e.) and moreover
	\[
	(\bm\pi_{F,H})_x=\frac{1}{f(x)}F(H_{\sigma\inv})_\ast\bm\pi_x\quad \bm\eta^\#-a.e.\ x\in X.
	\]
	For any Borel function $G\colon C([0,1];X)\times[0,1]\to \R$ we have that $\|G\|_{L^\infty((\bm\pi_{F,H})_x)}=\|G\|_{L^\infty((H_{\sigma\inv})_\ast\bm\pi_x)}$ for $\bm\eta^\#$-a.e. $x\in X$, thus we may assume that $F\equiv 1$.
	
	Now suppose that $G$ satisfies \eqref{eq:param_invariance}. Then $G\circ H_{\sigma\inv}=H$ $\bm\pi$-a.e. which readily implies that
	\[
	\|G\|_{L^\infty((H_{\sigma\inv})_\ast\bm\pi_x)}=\|G\circ H_{\sigma\inv}\|_{L^\infty(\bm\pi_x)}=\|G\|_{L^\infty(\bm\pi_x)}\quad \bm\eta^\#-a.e.\ x\in X.
	\]
	This completes the proof.
\end{proof}

\subsection{Measurability of $|Df_x|(y)$ and $|Df^y|(x)$}

\begin{prop}\label{prop:borel_slice}
    Let $f\in J^{1,p}(X,Y)$. For $\mu$-a.e. $x\in X$ there exists a representative of $\ud f_x\in \Gamma_p(T^*Y)$ so that $(x,y)\mapsto \ud_yf_x$ is Borel measurable. Similarly for $\nu$-a.e. $y\in Y$ there exists a representative of $\ud f^y\in\Gamma_p(T^*X)$ so that $(x,y)\mapsto \ud_xf^y$ is Borel.
\end{prop}
\begin{proof}

Let $(U,\varphi)$ and $(V,\varphi)$ be $p$-weak charts of dimension $N$ and $M$, respectively. Let $E\subset U$, $F\subset V$ be null-sets such that the pointwise norms $\Phi^x$ and $\Psi^y$ are well-defined and $f^y\in W^{1,p}(X)$, $f_x\in W^{1,p}(Y)$ whenever $x\notin E$ and $y\notin V$. We may define $\ud_xf^y$ as the unique vector $\xi\in (\R^N)^*$ for which
\[
\left\|\frac{\xi((\varphi\circ\alpha)_t')-(f^y_\alpha)'(t)}{|\alpha_t'|}\right\|_{L^\infty(\bm\pi_x)}=0,
\]
whenever this exists, and $0$ otherwise. Here $f^y_\alpha$ is the absolutely continuous representative of $f^y\circ\alpha$ if this exists  and 0 otherwise. It follows that $U\times V\ni(x,y)\mapsto \ud_xf^y$ thus defined is $\mu\times\nu$-measurable. A similar argument gives the claim for $\ud_yf_x$. Since $\mu\times\nu$ is Borel regular it follows that $\ud_xf^y$ and $\ud_yf_x$ have Borel representatives. By the arbitrariness of $(U,\varphi)$ and $(V,\psi)$ the claim follows.
\end{proof}
\begin{cor}
Let $f\in J^{1,p}(X,Y)$.  For $\mu$-a.e. $x\in X$ there exists a representative of $|Df_x|_p\in L^p(\nu)$ so that $(x,y)\mapsto |Df_x|_p(y)$ is Borel measurable. Similarly, for $\nu$-a.e. $y\in Y$ there exists a representative of $|Df^y|_p$ so that $(x,y)\mapsto |Df^y|_p(x)$ is Borel.
\end{cor}
\begin{proof}
The claim follows from Proposition \ref{prop:borel_slice} since $\Phi^x(\ud_xf^y)$ and $\Psi^y(\ud_yf_x)$ are Borel representatives of $|Df^y|_p$ and $|Df_x|_p$. 
\end{proof}
We  remark that the measurability of $|Df^y|_p(x)$ and $|Df_x|_p(y)$ can also be proven without using the $p$-weak differentiable structure.

\bibliographystyle{plain}
\bibliography{abib}
\end{document}